\documentclass[a4paper,12pt]{article}

\usepackage{authblk}
\usepackage{parskip}
\usepackage{mathtools}
\usepackage{amssymb}
\usepackage{amsmath}
\usepackage{latexsym}
\usepackage{mathrsfs}  
\usepackage{bbm}       
\usepackage{amsthm}    
\usepackage{relsize}   
\usepackage{graphicx}
\usepackage{thmtools}  
\usepackage{mathrsfs}  
\usepackage{paralist}  
\usepackage{lipsum}    
\usepackage[all]{xy}   

\numberwithin{equation}{section}

\usepackage{titlesec}   

\titleformat{\section}[block]{\bfseries\filcenter}
{{\upshape\thesection\enspace}}{.5em}{}

\titleformat{\subsection}[block]{\filcenter}
{{\upshape\thesubsection\enspace}}{.5em}{} 

\titleformat{\subsubsection}[block]{\filcenter}
{{\upshape\thesubsubsection\enspace}}{.5em}{} 

\usepackage{enumitem}  
\setlist{nosep}  
\setitemize[0]{leftmargin=*}
\setenumerate[0]{leftmargin=*}
\setenumerate[1]{label={(\arabic*)}} 

\newcommand{\N}{\mathbb{N}}     
\newcommand{\R}{\mathbb{R}}     
\newcommand{\C}{\mathbb{C}}     
\newcommand{\Prob}{\mathbb{P}}  
\newcommand{\Exp}{\mathbb{E}}   
\newcommand{\goth}[1]{\mathfrak{#1}} 
\newcommand{\inner}[2]{\left\langle #1 \, , \, #2 \right\rangle} 
\newcommand{\norm}[1]{\left|\left|#1\right|\right|}              
\newcommand{\triplet}[3]{\left( #1, #2, #3 \right) }             
\newcommand{\ProbSpace}{\triplet{\Omega}{\mathscr{F}}{\Prob}}    
\newcommand{\abs}[1]{\left| #1 \right|}                          

\newcommand{\defeq}{\mathrel{\mathop:}=}                         
\newcommand\restr[2]{{
  \left.\kern-\nulldelimiterspace 
  #1 
  \vphantom{\big|} 
  \right|_{#2} 
  }}
  
\makeatletter
\newsavebox{\@brx}
\newcommand{\llangle}[1][]{\savebox{\@brx}{\(\m@th{#1\langle}\)}%
  \mathopen{\copy\@brx\kern-0.5\wd\@brx\usebox{\@brx}}}
\newcommand{\rrangle}[1][]{\savebox{\@brx}{\(\m@th{#1\rangle}\)}%
  \mathclose{\copy\@brx\kern-0.5\wd\@brx\usebox{\@brx}}}
\makeatother

\theoremstyle{plain} 
\newtheorem{theorem}{Theorem}[section]    
\newtheorem{proposition}[theorem]{Proposition} 
\newtheorem{corollary}[theorem]{Corollary}

\theoremstyle{definition}

\newtheorem{remark}[theorem]{Remark}

 \title{\Large Almost Sure Uniform Convergence of Stochastic Processes in the Dual of a Nuclear Space}

\author{C. A. Fonseca-Mora}
\affil{  Escuela de Matem\'{a}tica, Universidad de Costa Rica, San Jos\'{e}, \\
Cod. 11501-2060, Costa Rica. \\
\noindent E-mail:  christianandres.fonseca@ucr.ac.cr }

\date{}

\begin{document}

 \maketitle

\abstract{Let $\Phi$ be a nuclear space and let $\Phi'$ denote its strong dual. In this paper we introduce sufficient conditions for the almost surely uniform convergence on bounded intervals of time for a sequence of $\Phi'$-valued processes having continuous (respectively c\`{a}dl\`{a}g) paths. The main result is formulated first in the general setting of cylindrical processes but later specialized to other situations of interest. In particular, we establish conditions for the convergence to occur in a Hilbert space continuously embedded in $\Phi'$. Furthermore, in the context of the dual of an ultrabornological nuclear space (like spaces of smooth functions and distributions) we also include applications to the convergence of a series of independent c\`{a}dl\`{a}g process and to the convergence of solutions to linear evolution equations driven by L\'{e}vy noise.}

\smallskip



\section{Introduction}

Let $\Phi$ be a nuclear space with strong dual $\Phi'$, and let $(X^{n}: n \in \N)$ be a sequence of $\Phi'$-valued process with continuous (respectively c\`{a}dl\`{a}g) paths.  Our main objective in this paper is to determine sufficient conditions for the existence of a $\Phi'$-valued process with continuous (respectively c\`{a}dl\`{a}g) paths such that $X^{n}$ converges to  $X$ almost surely uniformly on each bounded interval of time. 

The problem described in the above paragraph was firstly studied by I. Mitoma in \cite{Mitoma:1983AS}. There, Mitoma shown that if $\Phi$ is a Fr\'{e}chet nuclear space then a sufficient condition for the existence of an almost sure uniform limit is that for each $\phi \in \Phi$ the sequence $\inner{X^{n}}{\phi}$ converges almost surely uniformly on each bounded interval of time (Theorem 1 in \cite{Mitoma:1983AS}).

Motivated by the work of Mitoma and by the recent developments in \cite{FonsecaMora:Existence} on the existence of continuous and c\`{a}dl\`{a}g versions for cylindrical processes in the dual of a nuclear space, in this work we carried out a further extension to the work of Mitoma. Indeed, this extension is carried out in two directions: we assume that $\Phi$ is a general nuclear space and we consider cylindrical processes (i.e. a collection of cylindrical random variables indexed by time) instead of stochastic processes. Here the attribute cylindrical for a random variable refers to a random object whose law is a cylindrical, rather than a bona fide, probability measure. Recently, there has been an increasing interest in the usage of cylindrical processes as the driving noise of stochastic partial differential equations (see e.g.  \cite{FonsecaMora:SPDEs, KumarRiedle:2020,  SunXieXie:2020, vanNeervenVeraarWeis:2008, VeraarYaroslavtsev:2016}). 


In the next few paragraphs we describe our results. We start with some preliminaries in Sect. \ref{sectionPrelim}. Then  Section \ref{sectionAlmSureUnifConverg} is devoted to study our main objective as described above. Our main result is Theorem \ref{theoAlmosSureUnifConver} where we generalize the result of Mitoma in \cite{Mitoma:1983AS} to the context of a sequence of cylindrical stochastic processes defined in the strong dual to a general nuclear space. Section \ref{subSecMainResult} is devoted to the proof of Theorem \ref{theoAlmosSureUnifConver} and in Section \ref{subSectConvSingleHilbert} we prove an specialized version of Theorem \ref{theoAlmosSureUnifConver} for the convergence to occur in a Hilbert space continuously embedded in $\Phi'$ (Theorem \ref{theoAlmosSureUnifConverSingleHilbertSpace}). The proofs of these two theorems are based on a combination of the arguments used by Mitoma in \cite{Mitoma:1983AS} and the techniques developed by the author in \cite{FonsecaMora:Existence}. 

In Section \ref{sectUnifConStochProcess} we restrict our attention to specialized versions of Theorem \ref{theoAlmosSureUnifConver} for the case of a sequence of stochastic processes taking values in the dual to a ultrabornological nuclear space. Our main result is Theorem 
\ref{theoProcessesAlmosSureUnifConver} in Section \ref{subSectASUnifConvStochProcess} where we show that under the hypothesis that $\Phi$ is nuclear ultrabornological, if  $(X^{n}: n \in \N)$ is a sequence of  $\Phi'$-valued processes with continuous (respectively c\`{a}dl\`{a}g paths) and Radon probability distributions, then a sufficient condition for the existence of an almost sure uniform limit is that for each $\phi \in \Phi$ the sequence $\inner{X^{n}}{\phi}$ converges almost surely uniformly on each bounded interval of time. As a corollary we obtain the theorem originally proved by Mitoma in   
\cite{Mitoma:1983AS} (see Corollary \ref{coroMitomaTheorem}). 

Latter, in Section \ref{subSectConInLrHilbert} we introduce sufficient conditions for the $L^{r}$-convergence of the sequence $(X^{n}: n \in \N)$ in a Hilbert space embedded in $\Phi'$  uniformly on a bounded interval of time. Moreover, in Section \ref{subSectConvSeriesCadProcess} we apply our previous results to show that the  partial sums of a sequence of independent c\`{a}dl\`{a}d processes converges almost uniformly on bounded intervals of time provided we have convergence of finite dimensional distributions (see Theorem \ref{theoItoNisioSkorokhod}). This result extends to the context of duals of nuclear spaces a result from Basse'O-Connor and Rosi\'{n}ski \cite{BasseOConnorRosinski:2013} formulated for stochastic processes taking values in a separable Banach space. 

Finally, in Section \ref{sectConvLevySEE} we apply the tools developed in Section \ref{sectUnifConStochProcess} to provide sufficient conditions for the  almost surely convergence to the sequence of (generalized) Langevin equations 
$$ dY^{n}_{t}=A^{n}(t)'Y^{n} dt+ dL^{n}_{t},$$
where for each $n \in \N$, $(A^{n}(t)':t \geq 0)$ is the family of dual operators to a family $(A^{n}(t):t \geq 0)$ of continuous linear operators which generates a backward evolution system $(U(s,t): 0 \leq s \leq t < \infty)$ of continuous linear operators on $\Phi$ and $(L^{n}_{t}: t \geq 0)$ is a $\Phi'$-valued L\'{e}vy process. Existence, uniqueness and weak convergence of solutions in the Skorokhod space of such (generalized) Langevin equations were studied by the author in \cite{FonsecaMora:SEELevy}. 

\section{Preliminaries} \label{sectionPrelim}


Let $\Phi$ be a locally convex space (we will only consider vector spaces over $\R$). $\Phi$ is \emph{quasi-complete} if each bounded and closed subset of it is complete. $\Phi$ is called  \emph{bornological} (respectively \emph{ultrabornological}) if it is the inductive limit of a family of normed (respectively Banach)  spaces. A \emph{barreled space} is a locally convex space such that every convex, balanced, absorbing and closed subset is a neighborhood of zero. Every quasi-complete bornological space is ultrabornological and hence barrelled. For further details see \cite{Jarchow, NariciBeckenstein}.   

If $p$ is a continuous semi-norm on $\Phi$ and $r>0$, the closed ball of radius $r$ of $p$ given by $B_{p}(r) = \left\{ \phi \in \Phi: p(\phi) \leq r \right\}$ is a closed, convex, balanced neighborhood of zero in $\Phi$. A continuous seminorm (respectively norm) $p$ on $\Phi$ is called \emph{Hilbertian} if $p(\phi)^{2}=Q(\phi,\phi)$, for all $\phi \in \Phi$, where $Q$ is a symmetric, non-negative bilinear form (respectively inner product) on $\Phi \times \Phi$. For any given continuous seminorm $p$ on $\Phi$ let $\Phi_{p}$ be the Banach space that corresponds to the completion of the normed space $(\Phi / \mbox{ker}(p), \tilde{p})$, where $\tilde{p}(\phi+\mbox{ker}(p))=p(\phi)$ for each $\phi \in \Phi$. We denote by  $\Phi'_{p}$ the Banach space dual to $\Phi_{p}$ and by $p'$ the corresponding dual norm. Observe that if $p$ is Hilbertian then $\Phi_{p}$ and $\Phi'_{p}$ are Hilbert spaces. If $q$ is another continuous seminorm on $\Phi$ for which $p \leq q$, we have that $\mbox{ker}(q) \subseteq \mbox{ker}(p)$ and the inclusion map from $\Phi / \mbox{ker}(q)$ into $\Phi / \mbox{ker}(p)$ has a unique continuous and linear extension that we denote by $i_{p,q}:\Phi_{q} \rightarrow \Phi_{p}$. Furthermore, we have the following relation: $i_{p}=i_{p,q} \circ i_{q}$.

We denote by $\Phi'$ the topological dual of $\Phi$ and by $\inner{f}{\phi}$ the canonical pairing of elements $f \in \Phi'$, $\phi \in \Phi$. Unless otherwise specified, $\Phi'$ will always be consider equipped with its \emph{strong topology}, i.e. the topology on $\Phi'$ generated by the family of semi-norms $( \eta_{B} )$, where for each $B \subseteq \Phi$ bounded, $\eta_{B}(f)=\sup \{ \abs{\inner{f}{\phi}}: \phi \in B \}$ for all $f \in \Phi'$.  

Let $p$ and $q$ be continuous Hilbertian semi-norms on $\Phi$ such that $p \leq q$. The space of continuous linear operators (respectively Hilbert-Schmidt operators) from $\Phi_{q}$ into $\Phi_{p}$ is denoted by $\mathcal{L}(\Phi_{q},\Phi_{p})$ (respectively $\mathcal{L}_{2}(\Phi_{q},\Phi_{p})$). We employ an analogous notation for operators between the dual spaces $\Phi'_{p}$ and $\Phi'_{q}$. 

Let us recall that a (Hausdorff) locally convex space $(\Phi,\mathcal{T})$ is called \emph{nuclear} if its topology $\mathcal{T}$ is generated by a family $\Pi$ of Hilbertian semi-norms such that for each $p \in \Pi$ there exists $q \in \Pi$, satisfying $p \leq q$ and the canonical inclusion $i_{p,q}: \Phi_{q} \rightarrow \Phi_{p}$ is Hilbert-Schmidt. Other equivalent definitions of nuclear spaces can be found in \cite{Pietsch, Treves}. 

Let $\Phi$ be a nuclear space. If $p$ is a continuous Hilbertian semi-norm  on $\Phi$, then the Hilbert space $\Phi_{p}$ is separable (see \cite{Pietsch}, Proposition 4.4.9 and Theorem 4.4.10, p.82). Now, let $( p_{n} : n \in \N)$ be an increasing sequence of continuous Hilbertian semi-norms on $(\Phi,\mathcal{T})$. We denote by $\theta$ the locally convex topology on $\Phi$ generated by the family $( p_{n} : n \in \N)$. The topology $\theta$ is weaker than $\mathcal{T}$. We  will call $\theta$ a (weaker) \emph{countably Hilbertian topology} on $\Phi$ and we denote by $\Phi_{\theta}$ the space $(\Phi,\theta)$ and by $\widehat{\Phi}_{\theta}$ its completion. The space $\widehat{\Phi}_{\theta}$ is a (not necessarily Hausdorff) separable, complete, pseudo-metrizable (hence Baire and ultrabornological; see Example 13.2.8(b) and Theorem 13.2.12 in \cite{NariciBeckenstein}) locally convex space and its dual space satisfies $(\widehat{\Phi}_{\theta})'=(\Phi_{\theta})'=\bigcup_{n \in \N} \Phi'_{p_{n}}$ (see \cite{FonsecaMora:Existence}, Proposition 2.4).

 
Throughout this work we assume that $\ProbSpace$ is a complete probability space and consider a filtration $( \mathcal{F}_{t} : t \geq 0)$ on $\ProbSpace$ that satisfies the \emph{usual conditions}, i.e. it is right continuous and $\mathcal{F}_{0}$ contains all subsets of sets of $\mathcal{F}$ of $\Prob$-measure zero. The space $L^{0} \ProbSpace$ of equivalence classes of real-valued random variables defined on $\ProbSpace$ will always be considered equipped with the topology of convergence in probability and in this case it is a complete, metrizable, topological vector space.

A \emph{cylindrical random variable}\index{cylindrical random variable} in $\Phi'$ is a linear map $X: \Phi \rightarrow L^{0} \ProbSpace$ (see \cite{FonsecaMora:Existence}). If $X$ is a cylindrical random variable in $\Phi'$, we say that $X$ is \emph{$n$-integrable} ($n \in \N$)  if $ \Exp \left( \abs{X(\phi)}^{n} \right)< \infty$, $\forall \, \phi \in \Phi$, and has \emph{zero-mean} if $ \Exp \left( X(\phi) \right)=0$, $\forall \phi \in \Phi$. The \emph{Fourier transform} of $X$ is the map from $\Phi$ into $\C$ given by $\phi \mapsto \Exp ( e^{i X(\phi)})$.

Let $X$ be a $\Phi'$-valued random variable, i.e. $X:\Omega \rightarrow \Phi'$ is a $\mathscr{F}/\mathcal{B}(\Phi')$-measurable map. For each $\phi \in \Phi$ we denote by $\inner{X}{\phi}$ the real-valued random variable defined by $\inner{X}{\phi}(\omega) \defeq \inner{X(\omega)}{\phi}$, for all $\omega \in \Omega$. The linear mapping $\phi \mapsto \inner{X}{\phi}$ is called the \emph{cylindrical random variable induced/defined by} $X$. We will say that a $\Phi'$-valued random variable $X$ is \emph{$n$-integrable} ($n \in \N$) if the cylindrical random variable induced by $X$ is \emph{$n$-integrable}. 
 
Let $J=\R_{+} \defeq [0,\infty)$ or $J=[0,T]$ for  $T>0$. We say that $X=( X_{t}: t \in J)$ is a \emph{cylindrical process} in $\Phi'$ if $X_{t}$ is a cylindrical random variable for each $t \in J$. Clearly, any $\Phi'$-valued stochastic processes $X=( X_{t}: t \in J)$ induces/defines a cylindrical process under the prescription: $\inner{X}{\phi}=( \inner{X_{t}}{\phi}: t \in J)$, for each $\phi \in \Phi$. 

If $X$ is a cylindrical random variable in $\Phi'$, a $\Phi'$-valued random variable $Y$ is called a \emph{version} of $X$ if for every $\phi \in \Phi$, $X(\phi)=\inner{Y}{\phi}$ $\Prob$-a.e. A $\Phi'$-valued process $Y=(Y_{t}:t \in J)$ is said to be a $\Phi'$-valued \emph{version} of the cylindrical process $X=(X_{t}: t \in J)$ on $\Phi'$ if for each $t \in J$, $Y_{t}$ is a $\Phi'$-valued version of $X_{t}$.  

For a $\Phi'$-valued process $X=( X_{t}: t \in J)$ terms like continuous, c\`{a}dl\`{a}g, purely discontinuous, adapted, predictable, etc. have the usual (obvious) meaning. 

A $\Phi'$-valued random variable $X$ is called \emph{regular} if there exists a weaker countably Hilbertian topology $\theta$ on $\Phi$ such that $\Prob( \omega: X(\omega) \in (\widehat{\Phi}_{\theta})')=1$. Furthermore, a $\Phi'$-valued process $Y=(Y_{t}:t \in J)$ is said to be \emph{regular} if $Y_{t}$ is a regular random variable for each $t \in J$. 

\section{Almost Sure Uniform Convergence for Cylindrical Processes}\label{sectionAlmSureUnifConverg}

\subsection{Main Result}\label{subSecMainResult}

The main result of this paper is the following theorem that establishes conditions for the almost sure uniform convergence of a sequence of (cylindrical) processes in $\Phi'$ on bounded intervals of time.  

\begin{theorem}\label{theoAlmosSureUnifConver}
Let $\Phi$ be a nuclear space and let $(X^{n}: n \in \N)$, with $X^{n} =(X^{n}_{t}: t \geq 0)$, be a sequence of cylindrical process in $\Phi'$ satisfying:
\begin{enumerate}
\item For each $\phi \in \Phi$ the real-valued process $X^{n}(\phi)=( X^{n}_{t}(\phi): t \geq 0)$ is continuous (respectively c\`{a}dl\`{a}g).
\item For every $n \in \N$ and $T > 0$, the family $( X^{n}_{t}: t \in [0,T] )$ of linear maps from $\Phi$ into $L^{0} \ProbSpace$ is equicontinuous.  
\item For every $\phi \in \Phi$ and $T>0$, the sequence $X^{n}(\phi)(\omega)$ converges uniformly on $[0,T]$  for $\Prob$-a.e. $\omega \in \Omega$. 
\end{enumerate}
Then, there exists a weaker countably Hilbertian topology $\vartheta$ on $\Phi$ and some  $(\widehat{\Phi}_{\vartheta})'$-valued continuous (respectively c\`{a}dl\`{a}g) processes  $Y= (Y_{t}: t \geq 0)$ and  $Y^{n}= (Y^{n}_{t}: t \geq 0)$, $n \in \N$,  such that 
\begin{enumerate}[label=(\roman*)]
\item For every $\phi \in \Phi$ and $n \in \N$, the real-valued processes $\inner{Y^{n}}{\phi}$ and $X^{n}(\phi)$ are indistinguishable.
\item For $\Prob$-a.e. $\omega \in \Omega$ and every $T>0$, there exists a $\vartheta$-continuous Hilbertian seminorm $p=p(\omega,T)$ on $\Phi$ such that $Y^{n}(\omega)$ converges to $Y(\omega)$ in $\Phi'_{p}$ uniformly on $[0,T]$. 
\end{enumerate}
Moreover, as  $\Phi'$-valued processes $Y$  and $Y^{n}$, $n \in \N$,  are continuous (respectively c\`{a}dl\`{a}g) processes and for $\Prob$-a.e. $\omega \in \Omega$, $Y^{n}(\omega)$ converges to $Y(\omega)$ in $\Phi'$ uniformly on $[0,T]$ for every $T>0$. 
\end{theorem}

\begin{remark}\label{remaTheoAlmosSureUnifConver}
Suppose that in Theorem \ref{theoAlmosSureUnifConver} we have that $(X^{n}: n \in \N)$ is a sequence of $\Phi'$-valued regular processes. 
In such a case, the conclusion $(i)$ implies that for each $n \in \N$, $Y^{n}$ is a continuous (respectively c\`{a}dl\`{a}g) version of $X^{n}$ (see Proposition 2.11 in \cite{FonsecaMora:Existence}). 
If moreover each $X^{n}$ has continuous (respectively c\`{a}dl\`{a}g) paths, then $(i)$ implies that $X^{n}$ and $Y^{n}$ are indistinguishable processes (see Proposition 2.12 in \cite{FonsecaMora:Existence}).  Then, according to $(ii)$ for $\Prob$-a.e. $\omega \in \Omega$ and every $T>0$, there exists a $\vartheta$-continuous Hilbertian seminorm $p=p(\omega,T)$ on $\Phi$ such that $X^{n}(\omega)$ converges to $Y(\omega)$ in $\Phi'_{p}$ uniformly on $[0,T]$ (hence the convergence also occurs in $\Phi'$).
\end{remark}

In order to prove Theorem \ref{theoAlmosSureUnifConver} we will need the following result which is of a more general nature. 

\begin{proposition}\label{propEquicontFourTransforms}
Let $Z^{n}=(Z^{n}: t \in [0,T])$ be a sequence of cylindrical processes in the dual $\Psi'$ of an ultrabornological space $\Psi$ such that for each $\phi \in \Psi$, $Z^{n}(\phi)$ is continuous for each $n \in \N$ and $\sup_{n} \sup_{0 \leq t \leq T} \abs{Z_{t}^{n}(\phi)}<\infty$ $\Prob$-a.e. $\omega \in \Omega$.  Assume moreover that for each $n \in \N$, $t \in [0,T]$, the mapping $Z^{n}_{t}: \Psi \rightarrow L^{0}\ProbSpace$ is continuous. Then, for every $\epsilon >0$ there exists a continuous seminorm $p$ on $\Psi$ such that
\begin{equation}\label{eqPropEquicontFourTransforms}
\int_{\Omega} \, \sup_{n \in \N} \sup_{0 \leq t \leq T} \abs{1-e^{iZ^{n}_{t}(\phi)}} \, d\Prob \leq \epsilon + 2 p(\phi), \quad \forall \, \phi \in \Psi. 
\end{equation}
\end{proposition}
\begin{proof}
We modify the arguments used in the proof of Proposition 5.4 in \cite{FonsecaMora:Skorokhod}. Define 


\begin{equation*}
 V(\phi)=\int_{\Omega} \, \frac{\sup_{n} \sup_{0 \leq t \leq T} \abs{Z_{t}^{n}(\phi)}}{1+\sup_{n} \sup_{0 \leq t \leq T} \abs{Z_{t}^{n}(\phi)}} \, d\Prob, \quad \forall \, \psi \in \Psi.
\end{equation*}
Assume for the moment that $V$ is continuous. Let $\epsilon >0$. From the continuity of the exponential function there exists $\delta_{1} >0$ such that $\abs{ 1-e^{ir}} \leq \frac{\epsilon}{2}$ whenever $\abs{r} < \delta_{1}$. By the continuity of $V$ there exists a continuous seminorm $p$ on $\Psi$ such that $V(\phi) \leq (\delta_{2})^{2}$ $\forall \phi \in B_{p}(1)$, where $\delta_{2} = \min\left\{\delta_{1}, \frac{-1+\sqrt{1+\epsilon}}{2} \right\}$. 

Now because $\sup_{n} \sup_{0 \leq t \leq T} \abs{1-e^{iZ^{n}_{t}(\phi)}} \leq 2$ for any $\phi \in \Phi$, then if $\phi \in B_{p}(1)^{c}$, we have 
$$ \int_{\Omega} \,\sup_{n} \sup_{0 \leq t \leq T} \abs{1-e^{iZ^{n}_{t}(\phi)}} d\Prob  \leq  2p(\phi)^{2}.$$

On the other hand, given $\phi \in \Psi$, let $\Gamma_{\phi} =\{  \omega \in \Omega: \sup_{n} \sup_{0 \leq t \leq T} \abs{Z_{t}^{n}(\phi)(\omega)} \leq \delta_{2} \}$. Then, we have for all $\forall \phi \in B_{p}(1)$  
\begin{eqnarray*}
\int_{\Omega} \,\sup_{n} \sup_{0 \leq t \leq T} \abs{1-e^{iZ^{n}_{t}(\phi)}} d\Prob  
& \leq & \int_{\Gamma_{\phi}} \,\sup_{n} \sup_{0 \leq t \leq T} \abs{1-e^{iZ^{n}_{t}(\phi)}} d\Prob  + 2 \Prob (\Gamma_{\phi}^{c}) \\
& \leq & \frac{\epsilon}{2} +2 \frac{(1+\delta_{2})}{\delta_{2}} V(\phi) \leq  \frac{\epsilon}{2} +2 \frac{\epsilon}{4}= \epsilon,
\end{eqnarray*}
From the above inequalities we obtain \eqref{eqPropEquicontFourTransforms}. 

Now we show that $V$ is continuous. Since $\Psi$ is ultrabornological, by Proposition 5.7 in \cite{FonsecaMora:Skorokhod} it is enough to prove that $V$ is a sequentially lower semicontinuous pseudo-seminorm on $\Psi$.  

To conclude that $V$ is is a pseudo-seminorm we must check that the following is satisfied: 
\begin{enumerate}
\item $V(\phi_{1}+\phi_{2}) \leq V(\phi_{1})+V(\phi_{2})$, for every $\phi_{1}, \phi_{2} \in \Psi$.  
\item For $\lambda \in \R$, $\abs{\lambda} \leq 1$ implies $V(\lambda \phi) \leq V(\phi)$, $\forall \, \phi  \in \Psi$.
\item If $\lambda_{m} \rightarrow 0$, then $V(\lambda_{m} \phi) \rightarrow 0$, $\forall \, \phi \in \Psi$. 
\item $V(\phi_{m}) \rightarrow 0$ implies $V(\lambda \phi_{m})\rightarrow 0$, $\forall \lambda \in \R$. 
\end{enumerate}

Properties (1) and (2) are immediate from the fact that the function $x \mapsto \frac{x}{1+x}$ is a subadditive and increasing. 
 
To prove (3), let $\lambda_{m} \rightarrow 0$ and $\phi \in \Psi$. Since $\sup_{n} \sup_{0 \leq t \leq T} \abs{Z_{t}^{n}(\phi)}<\infty$ $\Prob$-a.e., by the dominated convergence theorem we have $V(\lambda_{n} \phi) \rightarrow 0$. 

To prove (4), let $\lambda \in \R$ and let $(\phi_{m} )_{m \in \N} \subseteq \Psi$ such that $\displaystyle{\lim_{m \rightarrow \infty} V(\phi_{m})=0}$. Consider any subsequence $(\phi_{m_{k}} )_{k \in \N}$ of $(\phi_{m} )_{m \in \N}$. Then  $\displaystyle{\lim_{k \rightarrow \infty} V(\phi_{m_{k}})=0}$. Hence, for each $r \in \N$ there exists $\phi_{m_{k(r)}}$ such that 
\begin{equation*}
V(\phi_{m_{k(r)}})=\int_{\Omega} \frac{\sup_{n} \sup_{0 \leq t \leq T} \abs{Z^{n}_{t}(\phi_{m_{k(r)}})}}{1+\sup_{n} \sup_{0 \leq t \leq T} \abs{Z^{n}_{t}(\phi_{m_{k(r)}})}} d \Prob \leq \frac{1}{2^{r+2}}.  
\end{equation*} 
Then,  
\begin{eqnarray*}
\Prob \left( \sup_{n} \sup_{0 \leq t \leq T} \abs{Z^{n}_{t}(\phi_{m_{k(r)}})} > 2^{-r} \right) 
& \leq & \frac{1+2^{-r}}{2^{-r}} \int_{\Omega} \frac{\sup_{n} \sup_{0 \leq t \leq T} \abs{Z^{n}_{t}(\phi_{m_{k(r)}})}}{1+\sup_{n} \sup_{0 \leq t \leq T} \abs{Z^{n}_{t}(\phi_{m_{k(r)}})}} d \Prob  \\
& \leq &  3 \cdot \frac{1}{2^{r+2}} \leq \frac{1}{2^{r}} 
\end{eqnarray*}
Then, it follows that for every $r \in \N$ we have
\begin{eqnarray*}
V(\lambda \phi_{m_{k(r)}}) 
& \leq & \Prob \left( \sup_{n} \sup_{0 \leq t \leq T} \abs{Z^{n}_{t}(\lambda \phi_{m_{k(r)}})} > \abs{\lambda} 2^{-r} \right)  + \frac{\abs{\lambda}2^{-r}}{1+\abs{\lambda}2^{-r}} \\
& < & 2^{-r} ( 1+\abs{\lambda}).
\end{eqnarray*}
So, we conclude that $\displaystyle{\lim_{r \rightarrow \infty} V(\lambda \phi_{m_{k(r)}})=0}$. As each subsequence of $(V(\lambda \phi_{m}): m \in \N)$ has a further subsequence that converges to $0$, it follows that $\displaystyle{\lim_{m \rightarrow \infty} V(\lambda \phi_{m})=0}$.  

Since we have shown that $V$ is a pseudo-seminorm, our final task is to show it is sequentially lower semicontinuous. To do this, suppose $(\phi_{m}: m \in \N)$ converges to $\phi$ in $\Psi$. Our assumption that each mapping $Z^{n}_{t}: \Psi \rightarrow L^{0}\ProbSpace$ is continuous implies that the mapping $\varphi \mapsto \sup_{n}  \sup_{0 \leq t \leq T} \abs{Z^{n}_{t}(\varphi)}$ is lower semicontinuous. Then, from Fatou's lemma we have
\begin{eqnarray*}
V(\phi) & \leq & \int_{\Omega} \, \liminf_{m \rightarrow \infty} \frac{\sup_{n} \sup_{0 \leq t \leq T}  \abs{Z_{t}^{n}(\phi_{m})}}{1+\sup_{n} \sup_{0 \leq t \leq T} \abs{Z_{t}^{n}(\phi_{m})}} \, d\Prob \\
& \leq & \liminf_{m \rightarrow \infty} \int_{\Omega} \, \frac{\sup_{n} \sup_{0 \leq t \leq T}  \abs{Z_{t}^{n}(\phi_{m})}}{1+\sup_{n} \sup_{0 \leq t \leq T} \abs{Z_{t}^{n}(\phi_{m})}} \, d\Prob \\
& \leq & \liminf_{m \rightarrow \infty} V(\phi_{m}). 
\end{eqnarray*}
We therefore conclude that $V$ is sequentially lower  semicontinuous thus continuous by the arguments given above. 
\end{proof}

\begin{proof}[Proof of Theorem \ref{theoAlmosSureUnifConver}] We will start by showing that in order to prove Theorem \ref{theoAlmosSureUnifConver} we only need to show its conclusions holds for a sequence of cylindrical processes $(X^{n}_{t}: t \in [0,T])$ defined on a bounded interval of time $[0,T]$ and under the assumption that each $X^{n}(\phi)$ is a continuous process. The arguments are similar to those used in the proof of Theorem 3.2 in \cite{FonsecaMora:Existence}, so we summarize the main steps. 

In effect, if the result is valid for every $T>0$, then for every $k \in \N$ we can find a weaker countably Hilbertian topology $\vartheta_{k}$ on $\Phi$, and some $(\widehat{\Phi}_{\vartheta_{k}})'$-valued continuous  processes $Y^{(k)}=\left( Y^{(k)}_{t}: t \in [0,k] \right)$ and $Y^{(k,n)}= \left(Y^{(k,n)}_{t}: t \in [0,k] \right)$, $n \in \N$, satisfying $(i)$ and $(ii)$ in Theorem \ref{theoAlmosSureUnifConver} on $[0,k]$. 

Let $\vartheta$ denote the countably Hilbertian topology on $\Phi$ generated by the families of seminorms generating the topologies $\vartheta_{k}$, $k \in \N$. The topology $\vartheta$ is finner than each $\vartheta_{k}$, but is weaker than the given topology on $\Phi$. Therefore, through the canonical inclusion from  $(\widehat{\Phi}_{\vartheta_{k}})'$ into  $(\widehat{\Phi}_{\vartheta})'$, each $Y^{(k,n)}$ can be considered as a $(\widehat{\Phi}_{\vartheta})'$-valued continuous process and similarly for $Y^{(k)}$.  Since for each $k, n \in \N$ and $\phi \in \Phi$, the processes $\inner{Y^{(k,n)}}{\phi}$ and $X^{n}(\phi)$ are indistinguishable on the time interval $[0,k]$, then for each $\phi \in \Phi$, $\inner{Y^{(k,n)}}{\phi}$ and $\inner{Y^{(k+1,n)}}{\phi}$ are indistinguishable as processes defined on $[0,k]$.  Then, the fact that as a $\Phi'$-valued process, each $Y^{(k,n)}$ is a regular process with continuous trajectories implies that $Y^{(k,n)}$ and $Y^{(k+1,n)}$ are indistinguishable as processes defined on $[0,k]$ (see Proposition 2.12 in \cite{FonsecaMora:Existence}). Moreover, since $Y^{(k)}$ and $Y^{(k+1)}$ are the $\omega$-wise uniform limit on $[0,k]$ of the sequences $(Y^{(k,n)}: n \in \N)$ and $(Y^{(k+1,n)}: n \in \N)$, then by uniqueness of limits we have that $Y^{(k)}$ and $Y^{(k+1)}$ are indistinguishable on $[0,k]$ for each $k \in \N$.     
 
For every $n \in \N$, take $Y^{n}=(Y^{n}_{t}: t \geq 0)$ defined by the prescription $Y^{n}_{t}=Y^{(k,n)}_{t}$ if $t \in [0,k]$. In a similar way let $Y=(Y_{t}: t \geq 0)$ defined by the prescription $Y_{t}=Y^{(k)}_{t}$ if $t \in [0,k]$. From the arguments in the above paragraph it is clear that $Y$ and each $Y^{n}$, $n \in \N$, is a $(\widehat{\Phi}_{\vartheta})'$-valued process  with continuous trajectories satisfying $(i)$ and $(ii)$ in Theorem \ref{theoAlmosSureUnifConver}. Since the canonical inclusion from $(\widehat{\Phi}_{\vartheta})'$ into $\Phi'$ is continuous, we directly obtain that as $\Phi'$-valued processes, $Y$  and $Y^{n}$, $n \in \N$,  are continuous processes and for $\Prob$-a.e. $\omega \in \Omega$, $Y^{n}(\omega)$ converges to $Y(\omega)$ in $\Phi'$ uniformly on $[0,T]$ for every $T>0$. The same arguments can be applied in the c\`{a}dl\`{a}g version case. 


From now on we will fix $T>0$ and show that the conclusions of Theorem \ref{theoAlmosSureUnifConver} are valid for a sequence of cylindrical process defined on $[0,T]$  under the continuous version assumption. 

For every $n \in \N$, from the first two assumptions in Theorem \ref{theoAlmosSureUnifConver} and the regularization theorem (Theorem 3.2 in \cite{FonsecaMora:Existence}) there exists a weaker countably Hilbertian topology $\theta_{n}$ on $\Phi$ and a $(\widehat{\Phi}_{\theta_{n}})'$-valued continuous process $Y^{n}=(Y^{n}_{t}: t \in [0,T])$ such that $\inner{Y^{n}}{\phi}$ and $X^{n}(\phi)$ are indistinguishable processes for every $\phi \in \Phi$. Since for each $t \in [0,T]$ the mapping $X^{n}_{t}: \Phi \rightarrow L^{0}\ProbSpace$ is continuous and linear, then the linear mapping $Y^{n}_{t}: \widehat{\Phi}_{\theta_{n}} \rightarrow L^{0}\ProbSpace$ is continuous on $\Phi$, and $\Phi$ being dense in $\widehat{\Phi}_{\theta_{n}}$ implies that $Y^{n}_{t}$ is continuous on $\widehat{\Phi}_{\theta_{n}}$. 

Let $\theta$ denotes the countably Hilbertian topology on $\Phi$ generated by the families of seminorms generating the topologies $\theta_{n}$, $n \in \N$. By definition $\theta$ is weaker than the given topology on $\Phi$ and is finer than each $\theta_{n}$. Thus each $Y^{n}$ is a  
$(\widehat{\Phi}_{\theta})'$-valued continuous process and $Y^{n}_{t}$ is continuous as a linear operator from $\widehat{\Phi}_{\theta}$ into $L^{0}\ProbSpace$. 

Let $(\epsilon_{m}: m \in \N)$ be a sequence of positive numbers converging to $0$. By Proposition \ref{propEquicontFourTransforms} and since $\widehat{\Phi}_{\theta}$ is ultrabornological, there exists an increasing sequence $(p_{m}: m \in \N)$ of $\theta$-continuous Hilbertian seminorms on $\Phi$ such that 
\begin{equation}\label{eqSequenceYnEquicontFourTransforms}
\int_{\Omega} \, \sup_{n \in \N} \sup_{0 \leq t \leq T} \abs{1-e^{i\inner{Y^{n}_{t}}{\phi}}} \, d\Prob \leq \epsilon_{m} + 2 p_{m}(\phi), \quad \forall \, \phi \in \widehat{\Phi}_{\theta}. 
\end{equation}
Observe that since $\Phi$ is nuclear, and each $p_{m}$ is continuous on $\Phi$, we can choose an increasing sequence $(q_{m}: m \in \N)$  of continuous Hilbertian seminorms on $\Phi$,  such that $p_{m} \leq q_{m}$ and the inclusion $i_{p_{m}, q_{m}}$ is Hilbert-Schmidt for every $m \in \N$. If we denote by $\alpha$ the countably Hilbertian topology on $\Phi$ generated by the seminorms $(q_{m}: m \in \N)$ we then have that $\alpha$ is finer than $\theta$ and weaker than the nuclear topology in $\Phi$. Therefore, we can regard each $Y^{n}$ as a $(\widehat{\Phi}_{\alpha})'$-valued continuous process.    

Moreover, being $\widehat{\Phi}_{\alpha}$ separable, if we choose any countable dense subset $(\xi_{k}: k \in \N)$ of $\widehat{\Phi}_{\alpha}$, by the Schmidt orthogonalization procedure for each $m \in \N$ there exists a complete orthonormal system $( \phi_{j}^{q_{m}}: j \in \N) \subseteq \Phi$ of $\Phi_{q_{m}}$, such that 
\begin{equation} \label{eqDecompDenseSetInTermsOrtoBasis}
\xi_{k}= \sum_{j=1}^{k} a_{j,k,m} \, \phi_{j}^{q_{m}} + \varphi_{k,m}, \quad \forall \, k \in \N, 
\end{equation}
with $a_{j,k,m} \in \R$ and $\varphi_{k,m} \in \mbox{Ker}(q_{m})$, for each $j,k \in \N$. 
Then, for every $m \in \N$ and by using \eqref{eqSequenceYnEquicontFourTransforms},  following similar arguments to those used in the proof of Lemma 3.8 in \cite{FonsecaMora:Existence} (see also  Lemma 3.2 in \cite{Mitoma:1983SC}) it follows that for any $C>0$ we have
\begin{flalign*}
& \Prob \left( \sup_{n \in \N} \sup_{0 \leq t \leq T}  \sum_{k=1}^{\infty} \abs{\inner{Y^{n}_{t}}{\phi_{k}^{q_{m}}}}^{2} > C^{2} \right) \\
& \leq  \lim_{r \rightarrow \infty} \frac{\sqrt{e}}{\sqrt{e}-1}  
  \int_{ \Omega} \sup_{n \in \N} \sup_{0 \leq t \leq T} \left( 1- \exp  \frac{-1}{2C^{2}} \sum_{k=1}^{r} \abs{\inner{Y^{n}_{t}}{\phi_{k}^{q_{m}}}}^{2}   \right) d \Prob  \\
& \leq  \lim_{r \rightarrow \infty} \frac{\sqrt{e}}{\sqrt{e}-1}      \int_{\R^{r}}\int_{ \Omega} \sup_{n \in \N} \sup_{0 \leq t \leq T}   \abs{ 1- \exp  i \sum_{k=1}^{r} \frac{z_{k} \inner{Y^{n}_{t}}{\phi_{k}^{q_{m}}}}{2C^{2}}   } d \Prob \, \frac{e^{\frac{-\abs{z}^{2}}{2}}}{(2\pi)^{\frac{r}{2}}} dz \\
& \leq   \frac{\sqrt{e}}{\sqrt{e}-1} \left(\epsilon_{m} + \frac{2}{C^{2}} \sum_{k=1}^{\infty} p_{m}(\phi_{k}^{q_{m}})^{2}  \right) 
\end{flalign*}
Since $\sum_{k=1}^{\infty} p_{m}(\phi_{k}^{q_{m}})^{2}<\infty$ because  $i_{p_{m},q_{m}}$ is Hilbert-Schmidt, letting $C\rightarrow \infty$ and by considering the probability of the complement we conclude that 
$$ \Prob \left( \sup_{n \in \N} \sup_{0 \leq t \leq T}  \sum_{k=1}^{\infty} \abs{\inner{Y^{n}_{t}}{\phi_{k}^{q_{m}}}}^{2} < \infty \right) \geq 1- \frac{\sqrt{e}}{\sqrt{e}-1} \epsilon_{m}. $$
In a similar way we can conclude that 
$$ \Prob \left( \sup_{n \in \N} \sup_{0 \leq t \leq T}  \sum_{k=1}^{\infty} \abs{\inner{Y^{n}_{t}}{\varphi_{k,m}}}^{2} = 0 \right) \geq 1- \frac{\sqrt{e}}{\sqrt{e}-1} \epsilon_{m}. $$
From the two inequalities above we have that
\begin{equation}\label{eqProbGammaM}
\Prob \left( \Gamma_{m} \right) \geq 1- 2  \frac{\sqrt{e}}{\sqrt{e}-1} \epsilon_{m}, 
\end{equation}
where  
\begin{equation}\label{eqDefiSetGammaM}
\Gamma_{m}=\left\{ \omega: \, \sup_{n \in \N} \sup_{0 \leq t \leq T}  \sum_{k=1}^{\infty} \abs{\inner{Y^{n}_{t}}{\phi_{k}^{q_{m}}}}^{2} < \infty, \, \sup_{n \in \N} \sup_{0 \leq t \leq T}  \sum_{k=1}^{\infty} \abs{\inner{Y^{n}_{t}}{\varphi_{k,m}}}^{2} = 0 \right\}. 
\end{equation}
Now, following the proof of Lemma 2 in \cite{Mitoma:1983SC}, if $\phi \in \widehat{\Phi}_{\alpha}$
there exists a subsequence $(\xi_{k_{\nu}}: \nu \in \N)$ of $(\xi_{k}: k \in \N)$ which  $\alpha$-converges to $\phi$; in particular $\lim_{\nu \rightarrow \infty} q_{m}(\phi-\xi_{k_{\nu}})=0$ $\forall m \in \N$. Then, because $Y^{n}_{t}(\omega) \in (\widehat{\Phi}_{\alpha})'$, if $\omega \in \Gamma_{m}$ by \eqref{eqDecompDenseSetInTermsOrtoBasis} we have 
\begin{eqnarray*}
\sup_{n \in \N} \sup_{0 \leq t \leq T} 
\abs{\inner{Y^{n}_{t}}{\phi}} 
& \leq & \sup_{n \in \N} \sup_{0 \leq t \leq T} \lim_{\nu \rightarrow \infty} \abs{ \sum_{j=1}^{k} a_{j,k_{\nu},m}  \inner{Y^{n}_{t}}{\phi_{j}^{q_{m}}} + \inner{Y^{n}_{t}}{\varphi_{k_{\nu},m}}} \\
&\leq & \left(  \lim_{\nu \rightarrow \infty} \sum_{j=1}^{k} a_{j,k_{\nu},m}^{2}  \right)^{1/2} C(m,\omega)^{1/2} \\
& = & \lim_{\nu \rightarrow \infty} q_{m}(\xi_{k_{\nu}}) C(m,\omega)^{1/2} < \infty,
\end{eqnarray*}   
where $C(m,\omega)= \sup_{n \in \N} \sup_{0 \leq t \leq T}  \sum_{k=1}^{\infty} \abs{\inner{Y^{n}_{t}(\omega)}{\phi_{k}^{q_{m}}}}^{2} < \infty $. Hence we have
\begin{equation}\label{eqInclusionYnFiniteHilbert}
\Gamma_{m} \subseteq  \left\{ \sup_{n \in \N} \sup_{0 \leq t \leq T} q'_{m}\left( Y^{n}_{t} \right) < \infty \right\}. 
\end{equation}
We will use the inclusion in \eqref{eqInclusionYnFiniteHilbert} in our construction of the limit process $Y$. But before we will need some further preparations. In our construction we will benefit from the arguments used in the proof of Theorem 1 in \cite{Mitoma:1983AS}. 

For each $m \in \N$, observe that since $q_{m}$ is a continuous Hilbertian seminorm on the nuclear space $\Phi$, there exists a continuous Hilbertian seminorm $\varrho_{m}$ on $\Phi$ such that $q_{m} \leq \varrho_{m}$ and the canonical inclusion $i_{q_{m}, \varrho_{m}}$ is Hilbert-Schmidt. Moreover, for each $m \in \N$ let $(\phi_{j}^{\varrho_{m}}: j \in \N) \subseteq \Phi$ be a complete orthonormal system in $\Phi_{\varrho_{m}}$ and let $(f_{j}^{\varrho_{m}}: j \in \N)$ be a complete orthonormal system in $\Phi'_{\varrho_{m}}$ dual to $(\phi_{j}^{\varrho_{m}}: j \in \N)$, i.e. $\inner{f_{j}^{\varrho_{m}}}{\phi_{j}^{\varrho_{m}}}=\delta_{ij}$, where $\delta_{ij}=1$ if $i=j$ and $\delta_{ij}=0$ if $i \neq j$. 

For each $m,j \in \N$, there exists a continuous real-valued process $X(\phi_{j}^{\varrho_{m}})$ which is the uniform limit on $[0,T]$ of the sequence $X^{n}(\phi_{j}^{\varrho_{m}})$ for $\Prob$-a.e. $\omega \in \Omega$. Since $\inner{Y^{n}}{\phi_{j}^{\varrho_{m}}}$ and $X^{n}(\phi_{j}^{\varrho_{m}})$ are indistinguishable, then  $\Prob (\Lambda_{m,j})=1$ where
\begin{equation}\label{eqDefiSetLamdaMJ}
\Lambda_{m,j}=\left\{ \omega \in \Omega \sup_{0 \leq t \leq T} \abs{\inner{Y^{n}_{t}(\omega)}{\phi_{j}^{\varrho_{m}}}- X_{t}(\phi_{j}^{\varrho_{m}})(\omega)  } \rightarrow 0, \quad \mbox{as} \quad n \rightarrow \infty \right\}. 
\end{equation}  
 
For each $m \in \N$, let $B_{m}=\Gamma_{m} \cap \bigcap_{j \in \N} \Lambda_{m,j}$. By \eqref{eqProbGammaM} we have 
$$ \Prob \left( B_{m} \right) \geq 1- 2  \frac{\sqrt{e}}{\sqrt{e}-1} \epsilon_{m}. $$
Since $\epsilon_{m} \rightarrow 0$, we conclude that $\Prob \left( \bigcup_{m} B_{m} \right)=1$. Set $\Omega_{1}=B_{1}$, and $\Omega_{n}=B_{n} \setminus B_{n-1}$ for $n \geq 2$. Then  $\Prob \left( \bigcup_{m} \Omega_{m} \right)=1$. 

We can now define the limit process. For each $t \in [0,T]$, we define
\begin{equation}\label{eqDefiLimitProcess}
Y_{t}(\omega)=
\begin{cases}
\sum_{j =1}^{\infty} X_{t}(\phi_{j}^{\varrho_{m}})(\omega) f_{j}^{\varrho_{m}}, & \mbox{if } \omega \in \Omega_{m}, \\
0, & \mbox{if } \omega \notin \bigcup_{m} \Omega_{m}. 
\end{cases}
\end{equation}
We will check that $Y$ is a well-defined process with continuous trajectories. We denote by $C(T, \Phi'_{\varrho_{m}})$ the Banach space of all the continuous mappings $F:[0,T] \rightarrow \Phi'_{\varrho_{m}}$ equipped with the norm $F \mapsto \sup_{0 \leq t \leq T} \varrho'_{m}(F(t))$ of uniform convergence on $[0,T]$. 

Let $\omega \in \Omega_{m}$. Since $\omega \in B_{m}$, then for all $k \in \N$ we have
$$ t \mapsto \sum_{j=1}^{k}  X_{t}(\phi_{j}^{\varrho_{m}})(\omega) f_{j}^{\varrho_{m}}  \in C(T, \Phi'_{\varrho_{m}}).$$
Moreover, for $l  > k \geq 1$ we have by \eqref{eqInclusionYnFiniteHilbert}
\begin{flalign}
& \sup_{0 \leq t \leq T} \varrho'_{m} \left(  \sum_{j=1}^{l}  X_{t}(\phi_{j}^{\varrho_{m}})(\omega) f_{j}^{\varrho_{m}} -  \sum_{j=1}^{k}  X_{t}(\phi_{j}^{\varrho_{m}})(\omega) f_{j}^{\varrho_{m}}  \right)^{2} \nonumber \\
& \leq   \sum_{j=k+1}^{l} \sup_{0 \leq t \leq T} \abs{ X_{t}(\phi_{j}^{\varrho_{m}})(\omega)}^{2} \nonumber \\
& = \sum_{j=k+1}^{l} \sup_{0 \leq t \leq T}  \liminf_{n \rightarrow \infty}  \abs{ \inner{Y^{n}_{t}(\omega)}{\phi_{j}^{\varrho_{m}}}}^{2}   \nonumber \\
& \leq  \sum_{j=k+1}^{l} \sup_{n} \sup_{0 \leq t \leq T} q'_{m}(Y^{n}_{t}(\omega))^{2} \, q_{m}(\phi_{j}^{\varrho_{m}})^{2} \nonumber \\
& \leq \left( \sup_{n} \sup_{0 \leq t \leq T}  q'_{m}(Y^{n}_{t}(\omega)) \right)^{2} \sum_{j=k+1}^{l} q_{m}(\phi_{j}^{\varrho_{m}})^{2}  \rightarrow 0.  \label{eqConvDefiProcessLimit}
\end{flalign}
Since $C(T, \Phi'_{\varrho_{m}})$ is complete, the Cauchy sequence $\sum_{j=1}^{k}  X_{t}(\phi_{j}^{\varrho_{m}})(\omega) f_{j}^{\varrho_{m}}$, $k \in \N$, converges and its limit, which corresponds to \eqref{eqDefiLimitProcess}, is an element of $C(T, \Phi'_{\varrho_{m}})$. 

Let $\vartheta$ be the countably Hilbertian topology on $\Phi$ generated by the family $(\varrho_{m}: m \in \N)$. Since the canonical inclusion from $\Phi'_{\varrho_{m}}$ into $(\widehat{\Phi}_{\vartheta})'$ is continuous, then $Y$ is a $(\widehat{\Phi}_{\vartheta})'$-valued process with continuous trajectories. Observe moreover that since the topology $\vartheta$ is finner than $\alpha$, then each $Y^{n}$ can be regarded as a $(\widehat{\Phi}_{\vartheta})'$-valued continuous process. 

Now we will check that the processes $(Y^{n}: n \in \N)$ converges to $Y$ in the sense described in the statement of Theorem \ref{theoAlmosSureUnifConver}. Let $\omega \in \Omega_{m}$. As before we have
\begin{flalign*}
& \sum_{j=1}^{\infty} \sup_{n} \sup_{0 \leq t \leq T} \abs{\inner{Y^{n}_{t}(\omega) -Y_{t}(\omega)}{\phi_{j}^{\varrho_{m}}}}^{2} \\
& \leq   4 \sum_{j=1}^{\infty} \sup_{n} \sup_{0 \leq t \leq T} \abs{\inner{Y^{n}_{t}}{\phi_{j}^{\varrho_{m}}}}^{2} \\
& \leq  4 \left( \sup_{n} \sup_{0 \leq t \leq T}  q'_{m}(Y^{n}_{t}(\omega)) \right)^{2} \sum_{j=1}^{\infty} q_{m}(\phi_{j}^{\varrho_{m}})^{2} \\
& <  \infty. 
\end{flalign*}
Then, by the dominated convergence theorem we have
\begin{flalign}
& \lim_{n \rightarrow \infty} \sup_{0 \leq t \leq T} \varrho'_{m} \left( Y^{n}_{t}(\omega) -Y_{t}(\omega) \right)^{2} \nonumber \\
& =  \lim_{n \rightarrow \infty} \sup_{0 \leq t \leq T} \sum_{j=1}^{\infty} \abs{\inner{Y^{n}_{t}(\omega) -Y_{t}(\omega)}{\phi_{j}^{\varrho_{m}}}}^{2} \nonumber \\
& \leq  \sum_{j=1}^{\infty} \lim_{n \rightarrow \infty} \sup_{0 \leq t \leq T} \abs{\inner{Y^{n}_{t}(\omega) -Y_{t}(\omega)}{\phi_{j}^{\varrho_{m}}}}^{2} \nonumber \\
& \leq  \sum_{j=1}^{\infty} \lim_{n \rightarrow \infty} \sup_{0 \leq t \leq T} \abs{\inner{Y^{n}_{t}(\omega)}{\phi_{j}^{\varrho_{m}}}-X_{t}(\phi_{j}^{\varrho_{m}})(\omega)}^{2} =  0. \label{eqUnifConvHilbSpace} 
\end{flalign}
Finally, since the canonical inclusion from $(\widehat{\Phi}_{\vartheta})'$ into $\Phi$ is continuous, we conclude that $Y$ and $Y^{n}$, $n \in \N$,  are continuous processes in $\Phi'$ and that $Y^{n}(\omega)$ converges to $Y(\omega)$ in $\Phi'$ uniformly on $[0,T]$ for $\Prob$-a.e. $\omega \in \Omega$.  
\end{proof}

\subsection{Convergence in a Single Hilbert Space}\label{subSectConvSingleHilbert}

Under some additional assumptions in Theorem \ref{theoAlmosSureUnifConver} we can strengthen its conclusions to show almost surely uniform convergence in a single Hilbert space $\Phi'_{\varrho}$ on every bounded interval of time $[0,T]$, where $\varrho$ is a continuous Hilbertian seminorm on $\Phi$  independent of $\omega$ and $T>0$. 

\begin{theorem}\label{theoAlmosSureUnifConverSingleHilbertSpace}
Let $\Phi$ be a nuclear space and let $(X^{n}: n \in \N)$, with $X^{n} =(X^{n}_{t}: t \geq 0)$, be a sequence of cylindrical process in $\Phi'$ satisfying:
\begin{enumerate}
\item For each $\phi \in \Phi$ the real-valued process $X^{n}(\phi)=( X^{n}_{t}(\phi):  t \geq 0)$ is continuous (respectively c\`{a}dl\`{a}g).
\item There exists a continuous Hilbertian seminorm $p$ on $\Phi$ such that for every $n \in \N$ and $t \geq 0$, the mapping $X^{n}_{t}: \Phi \rightarrow L^{0} \ProbSpace$ is $p$-continuous.  
\item For every $\phi \in \Phi$ and $T>0$, the sequence $X^{n}(\phi)(\omega)$ converges uniformly on $[0,T]$  for $\Prob$-a.e. $\omega \in \Omega$. 
\end{enumerate}
Then there exists a continuous Hilbertian seminorm $\varrho$ on $\Phi$, $p \leq \varrho$, such that $i_{p,\varrho}$ is Hilbert-Schmidt and   $\Phi'_{\varrho}$-valued continuous (respectively c\`{a}dl\`{a}g) processes  $Y= (Y_{t}: t \geq 0)$ and  $Y^{n}= (Y^{n}_{t}: t \geq 0)$, $n \in \N$,  such that 
\begin{enumerate}[label=(\roman*)]
\item For every $\phi \in \Phi$ and $n \in \N$, the real-valued processes $\inner{Y^{n}}{\phi}$ and $X^{n}(\phi)$ are indistinguishable.
\item For $\Prob$-a.e. $\omega \in \Omega$, $Y^{n}(\omega)$ converges to $Y(\omega)$ in $\Phi'_{\varrho}$ uniformly on $[0,T]$ for every $T>0$. 
\end{enumerate}
\end{theorem}
\begin{proof} 
It is enough to show that the conclusions of Theorem \ref{theoAlmosSureUnifConverSingleHilbertSpace} holds for a sequence of cylindrical processes $(X^{n}_{t}: t \in [0,T])$ defined on a bounded interval of time $[0,T]$ and under the assumption that each $X^{n}(\phi)$ is a continuous process. Indeed, in such a case we have for each $k \in \N$ the existence of $\Phi'_{\varrho}$-valued continuous processes $Y^{(k)}=\left( Y^{(k)}_{t}: t \in [0,k] \right)$ and $Y^{(k,n)}= \left(Y^{(k,n)}_{t}: t \in [0,k] \right)$, $n \in \N$, satisfying $(i)$ and $(ii)$  on $[0,k]$. Then for every $n \in \N$, take $Y^{n}=(Y^{n}_{t}: t \geq 0)$ defined by  $Y^{n}_{t}=Y^{(k,n)}_{t}$ if $t \in [0,k]$, and take $Y=(Y_{t}: t \geq 0)$ defined by the  $Y_{t}=Y^{(k)}_{t}$ if $t \in [0,k]$. As in the proof of Theorem \ref{theoAlmosSureUnifConver} it is easy to verify that $Y^{n}$ and $Y$ are $\Phi'_{\varrho}$-valued continuous processes satisfying (i) in (ii) in Theorem \ref{theoAlmosSureUnifConverSingleHilbertSpace}.   

Given $T>0$, we prove the result holds on the interval $[0,T]$. First, from assumptions (1) and (2) and from the version of the regularization theorem for $p$-continuous cylindrical processes  
(Theorem 4.1 in \cite{FonsecaMora:Existence}) there exists a continuous Hilbertian seminorm $r$ on $\Phi$, $p \leq r$, such that the inclusion $i_{p,r}$ is Hilbert-Schmidt, and such that for every $n \in \N$ there exists a $\Phi'_{r}$-valued continuous process $Y^{n}=(Y^{n}_{t}: t \in [0,T])$ satisfying that $\inner{Y^{n}}{\phi}$ and $X^{n}(\phi)$ are indistinguishable processes for every $\phi \in \Phi$. Since $\Phi'_{r}$ is a separable Hilbert space and $Y^{n}_{t}$ is $\Phi'_{r}$-valued then the mapping $Y^{n}_{t}: \Phi_{r} \rightarrow L^{0}\ProbSpace$ is continuous (see Theorem 2.10 in \cite{FonsecaMora:Existence}).

Let $(\epsilon_{m}: m \in \N)$ be a sequence of positive numbers converging to $0$. If we apply Proposition  \ref{propEquicontFourTransforms} to the mappings $Y^{n}$ on the space $\Phi_{r}$ then we can show that \eqref{eqSequenceYnEquicontFourTransforms} is satisfied for $p_{n}=r$ $\forall n \in \N$. Let $q$ be a continuous Hilbertian seminorm on $\Phi$ such that $r \leq q$ and $i_{r,q}$ is Hilbert-Schmidt. Being $\Phi_{q}$ separable there exists a complete orthonormal system $(\phi_{j}^{q}: j \in \N) \subseteq \Phi$ such that  \eqref{eqDecompDenseSetInTermsOrtoBasis}, \eqref{eqDefiSetGammaM} and \eqref{eqInclusionYnFiniteHilbert} holds for $q_{n}=q$ $\forall n \in \N$. If we choose any continuous Hilbertian seminorm $\varrho$ on $\Phi$ such that $q \leq \varrho$ and $i_{q, \varrho}$ is Hilbert-Schmidt. Then \eqref{eqDefiSetLamdaMJ}, \eqref{eqDefiLimitProcess}, \eqref{eqConvDefiProcessLimit} and \eqref{eqUnifConvHilbSpace} holds for $\varrho_{n}=\varrho$, $\forall n \in \N$. This shows that (i) and (ii) in Theorem \ref{theoAlmosSureUnifConverSingleHilbertSpace} holds on any bounded interval $[0,T]$. 
\end{proof}

\section{Uniform Convergence for Stochastic Processes}\label{sectUnifConStochProcess}

\subsection{Almost Surely Uniform Convergence}\label{subSectASUnifConvStochProcess}

In this section we restrict our attention to specialized versions of Theorem \ref{theoAlmosSureUnifConver} for the case of a sequence of stochastic processes taking values in the dual space $\Phi'$ to a ultrabornological nuclear space $\Phi$.

It is worth to mention that the class of ultrabornological nuclear spaces is large enough for many applications. It is well-known (see e.g. \cite{Pietsch, Schaefer, Treves}) that the space of test functions $\mathscr{E}_{K} \defeq \mathcal{C}^{\infty}(K)$ ($K$: compact subset of $\R^{d}$), $\mathscr{E}\defeq \mathcal{C}^{\infty}(\R^{d})$, the rapidly decreasing functions $\mathscr{S}(\R^{d})$, and the space of harmonic functions $\mathcal{H}(U)$ ($U$: open subset of $\R^{d}$; see \cite{Pietsch}, Section 6.3),  are all  examples of Fr\'{e}chet nuclear spaces. Their (strong) dual spaces $\mathscr{E}'_{K}$, $\mathscr{E}'$, $\mathscr{S}'(\R^{d})$, $\mathcal{H}'(U)$, are also nuclear spaces. On the other hand, the space of test functions $\mathscr{D}(U) \defeq \mathcal{C}_{c}^{\infty}(U)$ ($U$: open subset of $\R^{d}$), the space of polynomials $\mathcal{P}_{n}$ in $n$-variables, the space of real-valued sequences $\R^{\N}$ (with direct sum topology) are strict inductive limits of Fr\'{e}chet nuclear spaces (hence they are also nuclear). The space of distributions  $\mathscr{D}'(U)$  ($U$: open subset of $\R^{d}$) is also nuclear. All the above are examples of (complete) ultrabornological nuclear spaces.

Before we introduce results on almost surely convergence of stochastic processes in $\Phi'$, we establish our version of Theorem \ref{theoAlmosSureUnifConver} for the ultrabornological nuclear space setting. 


\begin{theorem}\label{theoUltraborAlmosSureUnifConver}
Let $\Phi$ be an ultrabornological nuclear space and let $(X^{n}: n \in \N)$, with  $X^{n} =(X^{n}_{t}: t \geq 0)$, be a sequence of cylindrical process in $\Phi'$ satisfying:
\begin{enumerate}
\item For each $\phi \in \Phi$ the real-valued process $X^{n}(\phi)=( X^{n}_{t}(\phi): t \geq 0)$ is continuous (respectively c\`{a}dl\`{a}g).
\item For every $n \in \N$ and $t \geq 0$, the mapping $X^{n}_{t}: \Phi \rightarrow L^{0} \ProbSpace$ is continuous.  
\item For every $\phi \in \Phi$ and $T>0$, the sequence $X^{n}(\phi)(\omega)$ converges uniformly on $[0,T]$  for $\Prob$-a.e. $\omega \in \Omega$. 
\end{enumerate}
Then, there exists a weaker countably Hilbertian topology $\vartheta$ on $\Phi$ and some  $(\widehat{\Phi}_{\vartheta})'$-valued continuous (respectively c\`{a}dl\`{a}g) processes  $Y= (Y_{t}: t \geq 0)$ and  $Y^{n}= (Y^{n}_{t}: t \geq 0)$, $n \in \N$,  such that 
\begin{enumerate}[label=(\roman*)]
\item For every $\phi \in \Phi$ and $n \in \N$, the real-valued processes $\inner{Y^{n}}{\phi}$ and $X^{n}(\phi)$ are indistinguishable.
\item For every $T>0$ and for $\Prob$-a.e. $\omega \in \Omega$, there exists a $\vartheta$-continuous Hilbertian seminorm $p=p(T,\omega)$ on $\Phi$ such that $Y^{n}(\omega)$ converges to $Y(\omega)$ in $\Phi'_{p}$ uniformly on $[0,T]$. 
\end{enumerate}
Moreover, as  $\Phi'$-valued processes $Y$  and $Y^{n}$, $n \in \N$,  are continuous (respectively c\`{a}dl\`{a}g) processes and for $\Prob$-a.e. $\omega \in \Omega$, $Y^{n}(\omega)$ converges to $Y(\omega)$ in $\Phi'$ uniformly on $[0,T]$ for every $T>0$. 
\end{theorem}
\begin{proof} The result follows from a minor modification of the proof of Theorem \ref{theoAlmosSureUnifConver} which we explain below. 

In effect, the hypothesis (2) in Theorem \ref{theoAlmosSureUnifConver} on equicontinuity for each $n \in \N$ and $T>0$ of the family $( X^{n}_{t}: t \in [0,T] )$ is only used in the application of the regularization theorem   (Theorem 3.2 in \cite{FonsecaMora:Existence}) to show the existence of a weaker countably Hilbertian topology $\theta_{n}$ on $\Phi$ and a $(\widehat{\Phi}_{\theta_{n}})'$-valued continuous process $Y^{n}=(Y^{n}_{t}: t \in [0,T])$ such that $\inner{Y^{n}}{\phi}$ and $X^{n}(\phi)$ are indistinguishable processes. However, if $\Phi$ is an ultrabornological nuclear space the same conclusions of the regularization theorem can be achieved by only assuming that each mapping $X^{n}_{t}: \Phi \rightarrow L^{0} \ProbSpace$ is continuous (see Corollary 3.11 in \cite{FonsecaMora:Existence}). The rest of the proof is similar to that of Theorem \ref{theoAlmosSureUnifConver}.
\end{proof}

We can apply Theorem \ref{theoUltraborAlmosSureUnifConver} to prove the following result on the almost uniform convergence for a sequence of stochastic processes taking values in the dual of an ultrabornological nuclear space.  

\begin{theorem}\label{theoProcessesAlmosSureUnifConver}
Let $\Phi$ be an ultrabornological nuclear space and let $(X^{n}: n \in \N)$, with  $X^{n} =(X^{n}_{t}: t \geq 0)$, be a sequence of  $\Phi'$-valued processes with continuous (respectively c\`{a}dl\`{a}g paths) with Radon probability distributions.  Assume further that for every $\phi \in \Phi$ and $T>0$, the sequence $\inner{X^{n}(\omega)}{\phi}$ converges uniformly on $[0,T]$  for $\Prob$-a.e. $\omega \in \Omega$. 

Then, there exists a $\Phi'$-valued process $Y= (Y_{t}: t \geq 0)$ with continuous (respectively c\`{a}dl\`{a}g paths) and Radon probability distributions, such that for $\Prob$-a.e. $\omega \in \Omega$ and each $T>0$, there exists a continuous Hilbertian seminorm $p=p(\omega,T)$ on $\Phi$ such that $X^{n}(\omega)$ converges to $Y(\omega)$ in $\Phi'_{p}$ uniformly on $[0,T]$.  
\end{theorem}
\begin{proof}
First we verify that assumptions (1)-(3) in Theorem \ref{theoUltraborAlmosSureUnifConver} are satisfied. Observe that (3) is part of our assumptions. Assumption (1) is immediate since each $X^{n}$  has continuous (respectively c\`{a}dl\`{a}g) paths. Assumption (2) follows because being $\Phi$ ultrabornological it is barrelled, and hence our assumption that $X_{t}$ has a Radon probability distribution implies that the mapping $X^{n}_{t}: \Phi \rightarrow L^{0} \ProbSpace$ is continuous (see Theorem 2.10 in \cite{FonsecaMora:Existence}). 

Then by Theorem \ref{theoUltraborAlmosSureUnifConver} there exists a weaker countably Hilbertian topology $\vartheta$ on $\Phi$ and some  $(\widehat{\Phi}_{\vartheta})'$-valued continuous (respectively c\`{a}dl\`{a}g) processes  $Y= (Y_{t}: t \geq 0)$ and  $Y^{n}= (Y^{n}_{t}: t \geq 0)$, $n \in \N$, satisfying (i) and (ii) in Corollary \ref{theoUltraborAlmosSureUnifConver}. 

Observe first that being a $(\widehat{\Phi}_{\vartheta})'$-valued process $Y$ is a regular $\Phi'$-valued process and hence each $Y_{t}$ has a Radon probability distribution (see Theorem 2.10 in \cite{FonsecaMora:Existence}).

Let $n \in \N$ and $t \geq 0$. Since $Y^{n}_{t}$ is a $(\widehat{\Phi}_{\vartheta})'$-valued random variable it is regular as a $\Phi'$-valued random variable. Similarly, since $X^{n}_{t}$ has a Radon probability distribution and $\Phi$ is barrelled, then $X^{n}_{t}$ is regular (see Theorem 2.10 in \cite{FonsecaMora:Existence}). Hence by (i) and since both $Y^{n}$ and $X^{n}$ have continuous (respectively c\`{a}dl\`{a}g) paths, then by Proposition 2.12 in \cite{FonsecaMora:Existence} they are indistinguishable processes. Then, by (ii) we have for each $T>0$ that there exists a  $\theta$-continuous Hilbertian seminorm $p$ on $\Phi$ such that $X^{n}(\omega)$ converges to $Y(\omega)$ in $\Phi'_{p}$ uniformly on $[0,T]$ for $\Prob$-a.e. $\omega \in \Omega$. 
\end{proof}

As a direct consequence of  Theorem \ref{theoProcessesAlmosSureUnifConver} we obtain the following result which was proved by I. Mitoma in \cite{Mitoma:1983AS} in the nuclear Fr\'{e}chet setting. 

\begin{corollary}[Mitoma's Theorem]\label{coroMitomaTheorem}
Let $\Phi$ be either a nuclear Fr\'{e}chet space or a countable inductive limit of Fr\'{e}chet nuclear spaces. If a sequence $(X^{n}: n \in \N)$ of $\Phi'$-valued processes with continuous (respectively c\`{a}dl\`{a}g paths) is such that for every $\phi \in \Phi$ and $T>0$, $\inner{X^{n}(\omega)}{\phi}$ converges uniformly on $[0,T]$  for $\Prob$-a.e. $\omega \in \Omega$. Then there exists a $\Phi'$-valued process $Y= (Y_{t}: t \geq 0)$ with continuous (respectively c\`{a}dl\`{a}g paths) such that for $\Prob$-a.e. $\omega \in \Omega$ and each $T>0$, there exists a continuous Hilbertian seminorm $p=p(\omega,T)$ on $\Phi$ such that $X^{n}(\omega)$ converges to $Y(\omega)$ in $\Phi'_{p}$ uniformly on $[0,T]$ for $\Prob$-a.e. $\omega \in \Omega$. 
\end{corollary}
\begin{proof}
If $\Phi$ is either a nuclear Fr\'{e}chet space or a countable inductive limit of Fr\'{e}chet nuclear spaces, then $\Phi$ is ultrabornological, and every Borel measure on $\Phi'$ is a Radon measure (see Corollary 1.3 of \cite{DaleckyFomin}, p.11). The result then follows from Theorem \ref{theoProcessesAlmosSureUnifConver}. 
\end{proof}

\subsection{Convergence in $L^{r}$ uniformly on bounded intervals}\label{subSectConInLrHilbert}

In the next result we introduce sufficient conditions for the convergence of a sequence of $\Phi'$-valued processes in $L^{r}(\Phi'_{\varrho})$ uniformly on a bounded interval of time $[0,T]$ for a continuous Hilbertian seminorm $\varrho$ on $\Phi$ (depending on $T>0$). 


\begin{theorem}\label{theoLrUniformConve}
Let $\Phi$ be a ultrabornological nuclear space and let $(X^{n}: n \in \N)$, with  $X^{n} =(X^{n}_{t}: t \geq 0)$, be a sequence of  $\Phi'$-valued processes with continuous (respectively c\`{a}dl\`{a}g paths)  with Radon probability distributions. 
Assume furthermore that: 
\begin{enumerate}
\item For every $\phi \in \Phi$ and $T>0$, the sequence $\inner{X^{n}(\omega)}{\phi}$ converges uniformly on $[0,T]$  for $\Prob$-a.e. $\omega \in \Omega$. 
\item There exists $r >1$ such that for every $\phi \in \Phi$ and $T>0$, 
$$\sup_{n \in \N}  \Exp \left( \sup_{0 \leq t \leq T} \abs{\inner{X^{n}_{t}}{\phi}}^{r} \right)< \infty.$$ 
\end{enumerate}
Then there exists a weaker countably Hilbertian topology $\vartheta$ on $\Phi$ and a $\Phi'$-valued process $Y= (Y_{t}: t \geq 0)$ with continuous (respectively c\`{a}dl\`{a}g) paths and Radon probability distributions, such that for each $T>0$ there exists a $\vartheta$-continuous Hilbertian seminorm $q=q(T)$ on $\Phi$ such that $\displaystyle{\lim_{n \rightarrow \infty} \Exp\left( \sup_{0 \leq t \leq T} q' (X^{n}_{t}-Y_{t})^{s} \right)=0}$ for every $1 \leq s<r$. 
\end{theorem}
\begin{proof}
We prove first that the result holds for a sequence of cylindrical processes $(X^{n}: n \in \N)$ defined on a bounded interval of time $[0,T]$ and under the assumption that each $(X^{n}_{t}: t \in [0,T])$ has continuous paths. 

Let $T>0$ and define $p: \Phi \rightarrow [0,\infty)$ by 
\begin{equation}\label{eqDefiSeminormUnifConvLr}
p(\phi)= \sup_{n \in \N} \left[ \Exp \left( \sup_{0 \leq t \leq T} \abs{\inner{X^{n}_{t}}{\phi}}^{r} \right) \right]^{1/r}, \quad \forall \phi \in \Phi. 
\end{equation}
By assumption (2) it is clear that $p(\phi)< \infty$ for each $\phi \in \Phi$. Moreover one can easily check that $p$ defines a seminorm on $\Phi$. We will prove that $p$ is continuous. We start by showing it is sequentially lower semicontinuous. 

In effect, let $(\phi_{k}: k \in \N)$ be a sequence converging to $\phi$  in $\Phi$, since for each $n \in \N$ the function $\varphi \mapsto  \sup_{0 \leq t \leq T} \abs{\inner{X^{n}_{t}}{\varphi}}^{r}$ is lower semicontinuous, then by Fatou's lemma we have that
\begin{eqnarray*}
p(\phi)^{r} & \leq & \sup_{n \in \N} \Exp \left( \liminf_{k \rightarrow \infty} \sup_{0 \leq t \leq T} \abs{\inner{X^{n}_{t}}{\phi_{k}}}^{r} \right) \\
& \leq & \sup_{n \in \N} \liminf_{k \rightarrow \infty} \Exp \left( \sup_{0 \leq t \leq T} \abs{\inner{X^{n}_{t}}{\phi_{k}}}^{r} \right) \\
& \leq & \liminf_{k \rightarrow \infty} p(\phi_{k})^{r}. 
\end{eqnarray*}
Then, $p$ is sequentially lower semicontinuous and because $\Phi$ is ultrabornological by Proposition 5.7 in \cite{FonsecaMora:Skorokhod} we have that $p$ is continuous. 

Now, from Markov and Jensen inequalities, for any $n \in \N$, $\epsilon >0$ and $\phi \in \Phi$, 
$$ \Prob \left( \omega : \sup_{0 \leq t \leq T} \abs{\inner{X^{n}_{t}(\omega)}{\phi}} > \epsilon \right) 
\leq \frac{1}{\epsilon} \left[  \Exp \left( \sup_{0 \leq t \leq T} \abs{\inner{X^{n}_{t}}{\phi}}^{r} \right) \right]^{1/r} \leq \frac{1}{\epsilon} p(\phi). $$
Hence for each $n \in \N$ and $t \in [0,T]$, the mapping $X^{n}_{t}: \Phi \rightarrow L^{0} \ProbSpace$ is $p$-continuous. Then by Theorem \ref{theoAlmosSureUnifConverSingleHilbertSpace} there exists a continuous Hilbertian seminorm $\varrho$ on $\Phi$, $p \leq \varrho$, such that $i_{p,\varrho}$ is Hilbert-Schmidt and  a $\Phi'_{\varrho}$-valued continuous  processes  $Y= (Y_{t}: t \in [0,T])$ such that for each $n \in \N$, $X^{n}$ has a $\Phi'_{\varrho}$-valued continuous indistinguishable version (which we denote again by $X^{n}$), and for $\Prob$-a.e. $\omega \in \Omega$, $X^{n}(\omega)$ converges to $Y(\omega)$ in $\Phi'_{\varrho}$ uniformly on $[0,T]$. 

Let $q$ be a continuous Hilbertian seminorm on $\Phi$ such that $\varrho \leq q$ and $i_{\varrho,q}$ is Hilbert-Schmidt. The dual operator $i'_{\varrho,q}: \Phi'_{\varrho} \rightarrow \Phi'_{q}$ is Hilbert-Schmidt and therefore is $r$-summing (see \cite{DiestelJarchowTonge}, Corollary 4.13, p.85). Then from the Pietsch domination theorem (see \cite{DiestelJarchowTonge}, Theorem 2.12, p.44) there exists a constant $C>0$ and a Radon probability measure $\nu$ on the unit ball $B^{*}_{\varrho}(1)$ of $\Phi_{\varrho}$ (equipped with the weak topology) such that, 
\begin{equation}\label{eqIneqPietschDomination}
q'(i'_{\varrho,q} f) \leq C \cdot \left( \int_{B^{*}_{\varrho}(1)} \, \abs{\inner{f}{\phi}}^{r} \, \nu(d\phi) \right)^{1/r} \quad \forall \, f \in \Phi'_{\varrho}.
\end{equation}
Now, for each $n \in \N$, since $X^{n}$ has a $\Phi'_{\varrho}$-valued continuous indistinguishable version and since $p \leq \varrho$, then it follows from \eqref{eqIneqPietschDomination} that 
\begin{eqnarray*}
\sup_{n \in \N} \Exp \left( \sup_{0 \leq t \leq T} q'(i'_{\varrho,q} X^{n}_{t})^{r}   \right) 
& \leq & C^{r} \sup_{n \in \N} \Exp \left( \sup_{0 \leq t \leq T} \, \int_{B^{*}_{\varrho}(1)} \, \abs{\inner{X^{n}_{t}}{\phi}}^{r} \, \nu(d\phi) \right) \\
& \leq & C^{r} \, \int_{B^{*}_{\varrho}(1)} \, \sup_{n \in \N}  \Exp \left( \sup_{0 \leq t \leq T} \abs{\inner{X^{n}_{t}}{\phi}}^{r} \right) \, \nu(d\phi)  \\
& = & C^{r} \, \int_{B^{*}_{\varrho}(1)} \, p(\phi)^{r} \, \nu(d\phi)  < \infty. 
\end{eqnarray*}
Thus if we identify $Y$ and each $X^{n}$ with their images in $\Phi'_{q}$ under the mapping  $i'_{\varrho,q}$ (which is injective), then we have $Y$ and each $X^{n}$ is a $\Phi'_{q}$-valued process with continuous paths, for $\Prob$-a.e. $\omega \in \Omega$, $X^{n}(\omega)$ converges to $Y(\omega)$ in $\Phi'_{q}$ uniformly on $[0,T]$, and by our calculations above we have 
$\sup_{n \in \N} \Exp \left( \sup_{0 \leq t \leq T} q'(X^{n}_{t})^{r}   \right) < \infty$. By an application of Fatou's lemma we can prove that  $\Exp \left( \sup_{0 \leq t \leq T} q'(Y_{t})^{r}   \right) < \infty$. Then we have for every $t \in [0,T]$ that $Y_{t},  X^{n}_{t} \in L^{r} (\Omega, \mathcal{F}, \Prob; \Phi'_{q})$. 

Let $1 \leq s<r$. For every $t \in [0,T]$ we have that $Y_{t},  X^{n}_{t} \in L^{s} (\Omega, \mathcal{F}, \Prob; \Phi'_{q})$. Moreover, since  
\begin{multline*}
\sup_{n \in \N} \Exp \left( \sup_{0 \leq t \leq T} q'(X^{n}_{t}-Y_{t})^{r}   \right) \\
\leq 2^{r} \left( \sup_{n \in \N} \Exp \left( \sup_{0 \leq t \leq T} q'(X^{n}_{t})^{r}   \right)
+  \Exp \left( \sup_{0 \leq t \leq T} q'(Y_{t})^{r}   \right) \right)
< \infty,
\end{multline*}
we have by Theorem 4.5.9 in (\cite{BogachevMT}, p.272) that the family $\displaystyle{\left( \sup_{0 \leq t \leq T} q'(X^{n}_{t}-Y_{t})^{s}:  n \in \N \right)}$ is uniformly integrable. Because $\displaystyle{ \lim_{n \rightarrow \infty} \sup_{0 \leq t \leq T} q'(X^{n}_{t}(\omega)-Y(\omega))^{s}=0}$ for $\Prob$-a.e. $\omega \in \Omega$, it then follows from Theorem 4.5.4 in (\cite{BogachevMT}, p.268) that $\displaystyle{\lim_{n \rightarrow \infty} \Exp\left( \sup_{0 \leq t \leq T} q' (X^{n}_{t}-Y_{t})^{s} \right)=0}$. This proves that the result holds in a bounded interval of time $[0,T]$. 

For the general case, let $(T_{k}: k \in \N)$ be an increasing sequence of positive numbers such that $\lim_{k \rightarrow \infty} T_{k} = \infty$. Then, for each $k \in \N$ there exists a continuous Hilbertian seminorm $q_{k}$ on $\Phi$ and a $\Phi'_{q_{k}}$-valued continuous process $Y^{(k)}= (Y^{(k)}_{t}: 0 \leq t \leq T_{k})$ with continuous paths such that $\displaystyle{\lim_{n \rightarrow \infty} \Exp\left( \sup_{0 \leq t \leq T_{k}} q' \left( X^{n}_{t}-Y^{(k)}_{t} \right)^{s} \right)=0}$ for every $1 \leq s<r$. 

We can assume without loss of generality that the sequence $(q_{k}: k \in N)$ is increasing. Let $\theta$ denotes the countably Hilbertian topology on $\Phi$ generated by the sequence $(q_{k}: k \in N)$. If we define $Y=(Y_{t}: t \geq 0)$ by the prescription $Y_{t}=Y^{(k)}_{t}$ if $t \in [0,T_{k}]$, then by uniqueness of limits of the convergence in each bounded interval of time we can conclude that $Y$ is a $(\widehat{\Phi}_{\vartheta})'$-valued  continuous process; hence a $\Phi'$-valued regular continuous process and thus  by Theorem 2.10 in \cite{FonsecaMora:Existence} each $Y_{t}$ has a Radon probability distribution. 

Finally, given $T>0$, if we choose $k \in \N$ such that $T_{k}>T$, then we have
$$\lim_{n \rightarrow \infty} \Exp\left( \sup_{0 \leq t \leq T} q' (X^{n}_{t}-Y_{t})^{s} \right) \leq \lim_{n \rightarrow \infty} \Exp\left( \sup_{0 \leq t \leq T_{k}} q' \left( X^{n}_{t}-Y^{(k)}_{t} \right)^{s} \right)=0,$$
for every $1 \leq s<r$.  
\end{proof}

\subsection{Convergence of Series of Independent C\`{a}dl\`{a}d Processes}\label{subSectConvSeriesCadProcess}

In this section we apply our result in Theorem \ref{theoProcessesAlmosSureUnifConver} to show that the  partial sums of a sequence of independent c\`{a}dl\`{a}d processes converges almost uniformly in bounded intervals of time provided we have convergence of finite dimensional distributions. Fundamental to our result is the work of Basse'O-Connor and Rosi\'{n}ski in \cite{BasseOConnorRosinski:2013}. 

\begin{theorem}\label{theoItoNisioSkorokhod}
Let $\Phi$ be a ultrabornological nuclear space and let $(X^{n}: n \in \N)$ be a sequence of independent $\Phi'$-valued c\`{a}dl\`{a}g processes  with Radon probability distributions. For every $n \in \N$, let $S^{n}=\sum_{k=1}^{n} X^{k}$. Assume that for every $\phi \in \Phi$, the following two conditions holds:
\begin{enumerate}
\item For every $n \in \N$, $\inner{X^{n}}{\phi}$ is symmetric. 
\item There exists a real-valued  c\`{a}dl\`{a}g process $Z^{\phi}$, such that for every $t_{1}, t_{2}, \dots, t_{m} \geq 0$, 
\begin{equation}\label{eqConvDistItoNisio}
(\inner{S^{n}_{t_{1}}}{\phi}, \dots, \inner{S^{n}_{t_{m}}}{\phi} ) \overset{d}{\rightarrow} (Z^{\phi}_{t_{1}}, \dots, Z^{\phi}_{t_{m}}).
\end{equation}
\end{enumerate}
Then, there exists  a $\Phi'$-valued c\`{a}dl\`{a}g process $S$ with Radon probability distributions such that for $\Prob$-a.e. $\omega \in \Omega$, $S^{n}(\omega)$ converges to $S(\omega)$ in $\Phi'$ uniformly on $[0,T]$ for every $T>0$. 
\end{theorem}
\begin{proof}
For each $\phi \in \Phi$, observe that $\inner{X^{n}}{\phi}$ is a sequence of independent real-valued c\`{a}dl\`{a}g processes  with symmetric probability distributions and $\inner{S^{n}}{\phi}=\sum_{k=1}^{n} \inner{X^{k}}{\phi}$. By our assumption in \eqref{eqConvDistItoNisio} and Theorem 2.1 in \cite{BasseOConnorRosinski:2013} there exists a real-valued c\`{a}dl\`{a}g process $S^{\phi}$ such that almost surely $\inner{S^{n}}{\phi} \rightarrow S^{\phi}$ uniformly on $[0,T]$ for every $T>0$. Here it is worth to remark that the although the result in \cite{BasseOConnorRosinski:2013} is formulated for  c\`{a}dl\`{a}g mappings defined on the time interval $[0,1]$, it can be extended to the case of c\`{a}dl\`{a}g mappings defined any bounded interval $[0,T]$ and therefore to the case of c\`{a}dl\`{a}g mappings defined on $[0, \infty)$. 

Now for each $n \in \N$, $S^{n}$ defines a $\Phi'$-valued c\`{a}dl\`{a}g process with Radon probability distributions (Section 4.2 in \cite{FonsecaMora:Skorokhod}). Then, by  Theorem \ref{theoProcessesAlmosSureUnifConver} there exists a $\Phi'$-valued c\`{a}dl\`{a}g process $S$ with Radon probability distributions such that for $\Prob$-a.e. $\omega \in \Omega$, $S^{n}(\omega)$ converges to $S(\omega)$ in $\Phi'$ uniformly on $[0,T]$ for every $T>0$. 
\end{proof}

\begin{remark}
Assume that the sequence $(X^{n}: n \in \N)$ is as in Theorem \ref{theoItoNisioSkorokhod} with the exception that for each $\phi \in \Phi$, each $\inner{X^{n}}{\phi}$ is symmetric. Then, in this case Theorem 2.1 in \cite{BasseOConnorRosinski:2013} shows that there exists a real-valued c\`{a}dl\`{a}g process $S^{\phi}$ and a sequence $y^{n,\phi}$ of real-valued c\`{a}dl\`{a}g processes with the property that $y^{n,\phi}_{t} \rightarrow 0$ for every $t \in [0,T]$, such that almost surely $\inner{S^{n}}{\phi} +y^{n,\phi} \rightarrow S^{\phi}$ uniformly on $[0,T]$. Define $Y^{n}_{t}(\phi)=y_{t}^{n,\phi}$ for each $n \in \N$ and $\phi \in \Phi$. If we where able to show that for each $t$ the mapping $\phi \mapsto Y^{n}_{t}(\phi)$ is linear and continuous, then by  the regularization theorem   (Theorem 3.2 in \cite{FonsecaMora:Existence}) we would have that each $Y^{n}$ determines a $\Phi'$-valued c\`{a}dl\`{a}g process with Radon distributions and then by the same arguments used in the proof of Theorem \ref{theoItoNisioSkorokhod} we would have that for $\Prob$-a.e. $\omega \in \Omega$, $S^{n}(\omega)+Y^{n}(\omega)$ converges to $S(\omega)$ in $\Phi'$ uniformly on $[0,T]$ for every $T>0$. 
\end{remark}


\section{Almost Sure Uniform Convergence of L\'{e}vy Processes and Stochastic Evolution Equations}\label{sectConvLevySEE}

Al through this section $\Phi$ denotes a nuclear space which is also quasi-complete and bornological (hence ultrabornological, by Theorem 13.2.12 in \cite{NariciBeckenstein}. 

\subsection{L\'{e}vy Processes and Stochastic Evolution Equations}

Recall from \cite{FonsecaMora:Levy} that a $\Phi'$-valued process $L=( L_{t} : t\geq 0)$ is called a \emph{L\'{e}vy process} if \begin{inparaenum}[(i)] \item  $L_{0}=0$ a.s., 
\item $L$ has \emph{independent increments}, i.e. for any $n \in \N$, $0 \leq t_{1}< t_{2} < \dots < t_{n} < \infty$ the $\Phi'$-valued random variables $L_{t_{1}},L_{t_{2}}-L_{t_{1}}, \dots, L_{t_{n}}-L_{t_{n-1}}$ are independent,  
\item L has \emph{stationary increments}, i.e. for any $0 \leq s \leq t$, $L_{t}-L_{s}$ and $L_{t-s}$ are identically distributed, and  
\item For every $t \geq 0$ the distribution $\mu_{t}$ of $L_{t}$ is a Radon measure and the mapping $t \mapsto \mu_{t}$ from $\R_{+}$ into the space $\goth{M}_{R}^{1}(\Phi')$ of Radon probability measures on $\Phi'$ is continuous at $0$ when $\goth{M}_{R}^{1}(\Phi')$  is equipped with the weak topology. \end{inparaenum} It is shown in Corollary 3.11 in \cite{FonsecaMora:Levy} that $L=( L_{t} : t\geq 0)$ has a regular, c\`{a}dl\`{a}g version $\tilde{L}=( \tilde{L}_{t} : t \geq 0)$ that is also a L\'{e}vy process. Moreover, there exists a weaker countably Hilbertian topology $\vartheta$ on $\Phi$ such that $\tilde{L}$ is a $(\widehat{\Phi}_{\vartheta})'$-valued c\`{a}dl\`{a}g process. We will therefore identify $L$ with $\tilde{L}$. 

In \cite{FonsecaMora:SEELevy}, the author introduced a theory of existence and uniqueness, path properties, and weak converge of linear stochastic evolution equations with L\'{e}vy noise. In the following paragraphs we summarize the main results we will need for our result in the next section. We start with some terminology. 

A family $(U(s,t): 0 \leq s \leq t < \infty) \subseteq \mathcal{L}(\Phi,\Phi)$ is a \emph{backward evolution system} if $U(t,t)=I$, $U(s,t)=U(s,r)U(r,t)$, $0 \leq s \leq r \leq t$. It is furthermore called \emph{strongly continuous} if for every $\phi \in \Phi$, $s,t \geq 0$, the mappings $[s,\infty) \ni r \mapsto U(s,r)\phi$ and $[0,t] \ni r \mapsto U(r,t)\phi$ are continuous. It is said that the family $A=(A_{t}:t \geq 0) \subseteq \mathcal{L}(\Phi,\Phi)$ \emph{generates} $(U(s,t): 0 \leq s \leq t < \infty)$ if the following forward and backward relations are satisfied: 
\begin{equation}\label{eqForwardEquationDeriv}
\frac{d}{dt} U(s,t)\phi = U(s,t)A(t) \phi, \quad \forall \, \phi \in \Phi, \, s \leq t.
\end{equation}  
\begin{equation} \label{eqBackwardEquationDeriv}
\frac{d}{ds} U(s,t)\phi = -A(s)U(s,t) \phi, \quad \forall \, \phi \in \Phi, \, s \leq t.
\end{equation}

The backward evolution system $(U(s,t): 0 \leq s \leq t < \infty)$ is called $(C_{0},1)$ if for each $T>0$ and each continuous seminorm $p$ on $\Phi$ there exists some $\vartheta_{p} \geq 0$ and a continuous seminorm $q$ on $\Phi$ such that  $ p(U(s,t)\phi) \leq e^{\vartheta_{p} (t-s)} q(\phi)$, for all $0 \leq s \leq t \leq T$ and $\phi \in \Phi$.  


Consider the following \emph{(generalized) Langevin equation} with L\'{e}vy noise: 
\begin{equation}\label{eqGenLanEquaLevy}
d Y_{t}=A(t)' Y_{t}dt + dL_{t}, \quad t \geq 0, 
\end{equation}
with initial condition $Y_{0}=\eta$ $\Prob$-a.e., where $\eta$ is a $\mathcal{F}_{0}$-measurable $\Phi'$-valued regular random variable. A \emph{weak solution} to \eqref{eqGenLanEquaLevy} is a $\Phi'$-valued regular adapted process $Y=(Y_{t}: t \geq 0)$ satisfying that for any given $t \geq 0$, for each $\phi \in \Phi$ we have $\int_{0}^{t} \abs{\inner{Y_{s}}{A(s)\phi}} ds< \infty$ $\Prob$-a.e. and  
\begin{equation*}
\inner{Y_{t}}{\phi}=\inner{\eta}{\phi} +\int_{0}^{t} \inner{Y_{s}}{A(s)\phi} ds+\inner{L_{t}}{\phi}. 
\end{equation*}

Suppose that the mapping $t \mapsto A(t) \psi$ is continuous from $[0,\infty)$ into $\Phi$ for every $\phi \in \Phi$. By Theorem 5.7 in \cite{FonsecaMora:SEELevy}, the generalized Langevin equation \eqref{eqGenLanEquaLevy} has a unique (up to indistinguishable versions) weak solution $(X_{t}: t \geq 0)$ which is regular and has c\`{a}dl\`{a}g paths, this solution satisfies $\Prob$-a.e.
\begin{equation*} \label{eqDefiCadlagVersMildSoluLevy}
\inner{X_{t}}{\phi}=\inner{\eta}{U(0,t)\phi}+\int_{0}^{t} \inner{L_{s}}{A(s)U(s,t)\phi} ds + \inner{L_{t}}{\phi}, \quad \forall \, t \geq 0, \, \phi \in \Phi.
\end{equation*} 
Moreover if $(L_{t}: t\geq 0)$ has continuous paths, then $(X_{t}: t \geq 0)$ has continuous paths too. For further details and examples the reader is referred to \cite{FonsecaMora:SEELevy}.

\subsection{Convergence of Solutions to Stochastic Evolution Equations}

In this section we prove some applications of our previous results to the convergence of L\'{e}vy processes and of linear stochastic evolution equations driven by L\'{e}vy noise. 

We start with the following result which shows that the collection of all the $\Phi'$-valued L\'{e}vy processes is closed under sequential limits under weakly (in the duality sense) almost surely uniform convergence. 

\begin{theorem}\label{theoSequeLimitLevy}
Let $\Phi$ be an ultrabornological nuclear space and let $(L^{n}: n \in \N)$, with  $L^{n} =(L^{n}_{t}: t \geq 0)$, be a sequence of  $\Phi'$-valued L\'{e}vy processes.  Assume further that for every $\phi \in \Phi$ and $T>0$, the sequence $\inner{L^{n}(\omega)}{\phi}$ converges uniformly on $[0,T]$  for $\Prob$-a.e. $\omega \in \Omega$. 

Then, there exists a $\Phi'$-valued L\'{e}vy process $Y= (Y_{t}: t \geq 0)$ such that for $\Prob$-a.e. $\omega \in \Omega$ and each $T>0$, there exists a continuous Hilbertian seminorm $p=p(\omega,T)$ on $\Phi$ such that $L^{n}(\omega)$ converges to $Y(\omega)$ in $\Phi'_{p}$ uniformly on $[0,T]$.  
\end{theorem}
\begin{proof}
The existence of a $\Phi'$-valued c\`{a}dl\`{a}g process $Y= (Y_{t}: t \geq 0)$ with Radon probability distributions which is the limit in the sense indicated in Theorem \ref{theoSequeLimitLevy} is a consequence of Theorem \ref{theoProcessesAlmosSureUnifConver}. So we only have to show that $Y$ is a $\Phi'$-valued L\'{e}vy process.  

In effect, for every $m \in \N$, $\phi_{1}, \dots, \phi_{m} \in \Phi$, the  convergence of the sequence $X^{n}$ to $Y$ shows that the sequence of $\R^{m}$-valued L\'{e}vy processes  $\left(\inner{L^{n}}{\phi_{1}}, \dots, \inner{L^{n}}{\phi_{m}} \right)$ converges almost uniformly on each bounded interval of time to $\left(\inner{Y}{\phi_{1}}, \dots, \inner{Y}{\phi_{m}} \right)$. Then $\left(\inner{Y}{\phi_{1}}, \dots, \inner{Y}{\phi_{m}} \right)$ is a $\R^{m}$-valued L\'{e}vy process (see \cite{ApplebaumLPSC}, Theorem 1.3.7). Hence $Y$ induces a cylindrical L\'{e}vy process in $\Phi'$. Moreover, since $Y$ has Radon probability  distributions and $\Phi$  is ultrabornological, by Theorem 2.10 and Proposition 3.10 in \cite{FonsecaMora:Existence} we have that for each $T>0$, the family of linear mappings $(Y_{t}: t \in [0,T])$ from $\Phi$ into $L^{0}\ProbSpace$ is equicontinuous. Hence, by Theorem 3.8 in  \cite{FonsecaMora:Levy} it follows that $Y$ has an indistinguishable version that is a $\Phi'$-valued L\'{e}vy process, thus $Y$ is a $\Phi'$-valued L\'{e}vy process. 
\end{proof}

Our next result exhibits sufficient conditions for the almost surely uniform convergence of the solutions of a sequence of linear evolution equations driven by L\'{e}vy noise.

For each $n=0,1,2,\dots$, let $L^{n}=(L^{n}_{t}: t \geq 0)$ be a $\Phi'$-valued L\'{e}vy process with c\`{a}dl\`{a}g paths, $(U^{n}(s,t): 0 \leq s \leq t)$ a backward evolution system with family of generators $(A^{n}(t): t \geq 0) \subseteq  \mathcal{L}(\Phi,\Phi)$ satisfying that the mapping $t \mapsto A^{n}(t) \phi$ is continuous from $[0,\infty)$ into $\Phi$ for every $\phi \in \Phi$, $\eta^{n}$ is a $\mathcal{F}_{0}$-measurable $\Phi'$-valued regular random variable.

As mentioned in the previous section, the generalized Langevin equation 
\begin{equation}\label{eqDefiSequenGenLangEqua}
dY^{n}_{t}=A^{n}(t)'Y^{n} dt+ dL^{n}_{t}, \quad Y^{n}_{0}=\eta^{n}_{0},
\end{equation} 
has a unique $\Phi'$-valued c\`{a}dl\`{a}g regular solution $X^{n}=(X^{n}_{t}: t \geq 0)$  which satisfies for each $t \geq 0$, $\phi \in \Phi$, 
\begin{equation} \label{eqDefiOrnUhlCadlagAlmsSureConverg}
\inner{X^{n}_{t}}{\phi}=\inner{\eta^{n}}{U^{n}(0,t)\phi}+\int_{0}^{t} \inner{L^{n}_{s}}{A^{n}(s)U^{n}(s,t)\phi} ds + \inner{L^{n}_{t}}{\phi}.
\end{equation}  
Sufficient conditions for almost surely uniform convergence of $X^{n}$ to $X$ are given below:

\begin{theorem}\label{theoAlmosSureConvOrnsUhleProce} Assume the following:
\begin{enumerate}
\item For $\Prob$-a.e. $\omega \in \Omega$, $\eta^{n}(\omega) \rightarrow \eta^{0}(\omega)$ in $\Phi'$. 
\item $U^{n}(0,t) \phi \rightarrow U^{0}(0,t) \phi$ as $n \rightarrow \infty$ for each $\phi \in \Phi$ uniformly in $t$ on bounded intervals of time, and $A^{n}(s) U^{n}(s,t) \phi \rightarrow A^{0}(s) U^{0}(s,t)\phi$ as $n \rightarrow \infty$ for each $\phi \in \Phi$ uniformly for $(s,t)$ in bounded intervals of time. 
\item For every $\phi \in \Phi$ and $T>0$, the sequence $\inner{L^{n}(\omega)}{\phi}$ converges uniformly to $\inner{L(\omega)}{\phi}$ as $n \rightarrow \infty$ on $[0,T]$  for $\Prob$-a.e. $\omega \in \Omega$. 
\end{enumerate}
Then for $\Prob$-a.e. $\omega \in \Omega$ and each $T>0$, there exists a continuous Hilbertian seminorm $p=p(\omega,T)$ on $\Phi$ such that $X^{n}(\omega)$ converges to $X^{0}(\omega)$ in $\Phi'_{p}$ uniformly on $[0,T]$.  
\end{theorem}
\begin{proof} We benefit from ideas used in the proof of Theorem 4.3 in \cite{PerezAbreuTudor:1992}.  Given $T>0$ and $\phi \in \Phi$, observe from \eqref{eqDefiOrnUhlCadlagAlmsSureConverg} that 
\begin{eqnarray*}
\inner{X^{n}_{t}-X^{0}_{t}}{\phi}
& = & \inner{\eta^{n}}{U^{n}(0,t)\phi}- \inner{\eta^{0}}{U^{0}(0,t)\phi} \\
& + & \int_{0}^{t} \inner{L^{n}_{s}}{A^{n}(s)U^{n}(s,t)\phi} ds- \int_{0}^{t} \inner{L^{0}_{s}}{A^{0}(s)U^{0}(s,t)\phi} ds \\
& + & \inner{L^{n}_{t}-L^{0}_{t}}{\phi}.
\end{eqnarray*}
We must check that the three terms in the right-hand side of the above equality all converge to $0$ as $n \rightarrow \infty$. 

First, we have clearly by $(3)$ that   $\sup_{0 \leq t \leq T} \abs{ \inner{L^{n}_{t}-L^{0}_{t}}{\phi}} \rightarrow 0$ as $n \rightarrow \infty$. 

Now, observe that by $(2)$ the set 
$$ B = \bigcup_{n} \bigcup_{0 \leq t \leq T} U^{n}(0,t) \phi, $$
is a bounded subset of $\Phi$. Moreover, by a direct consequence of the Banach-Steinhaus theorem (recall $\Phi$ is barrelled) for each $\omega \in \Omega$ there exists a continuous Hilbertian seminorm $p=p(\omega)$ on $\Phi$ such that $\eta^{0}(\omega) \in \Phi'_{p}$. Therefore,  by $(1)$ and $(2)$, for $\Prob$-a.e. $\omega \in \Omega$ and corresponding $p=p(\omega)$ we have
\begin{flalign*}
& \sup_{0 \leq t \leq T} \abs{ \inner{\eta^{n}}{U^{n}(0,t)\phi}- \inner{\eta^{0}}{U^{0}(0,t)\phi}  } \\
& \leq p'(\eta^{0}) \sup_{0 \leq t \leq T} p(U^{n}(0,t)\phi-U^{0}(0,t)\phi)  +  \sup_{\psi \in B} \abs{ \inner{\eta^{n}-\eta^{0}}{\psi}} \rightarrow 0, 
\end{flalign*}
as $n \rightarrow \infty$.  

We now need to prove the convergence of the integral term. To do this we will need to make some preparations. First, observe that by $(2)$ the set 
$$ K = \bigcup_{n} \bigcup_{0 \leq s \leq t \leq T} A^{n}(s) U^{n}(s,t) \phi, $$
is a bounded subset of $\Phi$. 

On the other hand, by $(3)$ and Theorem \ref{theoSequeLimitLevy}, for $\Prob$-a.e. $\omega \in \Omega$ we have $L^{n}(\omega) \rightarrow L^{0}(\omega)$ in $\Phi'$ uniformly on $[0,T]$. Furthermore,  for $\Prob$-a.e. $\omega \in \Omega$ by Theorem 3.10 in \cite{FonsecaMora:Levy} and Remark 3.9 in \cite{FonsecaMora:Existence}, there exists a continuous Hilbertian seminorm $p=p(\omega,T)$ on $\Phi$ such that the mapping $t \mapsto L^{0}_{t}(\omega)$ is c\`{a}dl\`{a}g from $[0,T]$ into $\Phi'_{p}$. 

Then, given the arguments from the above paragraphs and by $(2)$ we have for $\Prob$-a.e. $\omega \in \Omega$ with corresponding $p=p(\omega,T)$, that 
\begin{flalign*}
& \sup_{0 \leq t \leq T} \abs{\int_{0}^{t} \inner{L^{n}_{s}}{A^{n}(s)U^{n}(s,t)\phi} ds- \int_{0}^{t} \inner{L^{0}_{s}}{A^{0}(s)U^{0}(s,t)\phi} ds } \\
& \leq  T \sup_{0 \leq t \leq T} p'(L_{t}^{0}) \sup_{0 \leq s \leq t \leq T} p(A^{n}(s)U^{n}(s,t)\phi -A^{0}(s)U^{0}(s,t)\phi) \\
& \hspace{10pt} + T \sup_{0 \leq t \leq T} \sup_{\psi \in K} \abs{\inner{L^{n}_{t}-L^{0}_{t}}{\psi}} \rightarrow 0, 
\end{flalign*}
as $n \rightarrow \infty$. Then, by what we have proved above, for $\Prob$-a.e. $\omega \in \Omega$ we have $\sup_{0 \leq t \leq T} \abs{\inner{X^{n}_{t}-X^{0}_{t}}{\phi}} \rightarrow 0$ as $n \rightarrow \infty$. The conclusion of Theorem \ref{theoAlmosSureConvOrnsUhleProce} now follows from Theorem \ref{theoProcessesAlmosSureUnifConver}. 
\end{proof}


\begin{remark}
In the following paragraphs we describe some sufficient conditions for assumption (2) in Theorem \ref{theoAlmosSureConvOrnsUhleProce} to hold. These conditions were introduced in the works \cite{KallianpurPerezAbreu:1988, KallianpurPerezAbreu:1989, PerezAbreuTudor:1992} on the study of solutions to linear stochastic evolution equations driven by square integrable martingale noise. 

Let  $\Phi$ be a Fr\'{e}chet nuclear space. For each $n=0,1,2, \dots$, let  $(A^{n}(t): t \geq 0) \subseteq  \mathcal{L}(\Phi,\Phi)$ such that for each $t \geq 0$, $A^{n}(t)$ is the infinitesimal generator of a $(C_{0},1)$-semigroup $(S^{n}_{t}(s): s \geq 0)$ on $\Phi$ 
(see Section 4.1 in \cite{FonsecaMora:SEELevy} and references therein). 

Assume that there exists an increasing sequence $(q_{m}: m \geq 0)$ of norms generating the topology on $\Phi$ such that the following two conditions hold: 
\begin{enumerate}
\item For every $k \geq 0$, there exists $m \geq k$ such that, for each $t \geq 0$ and $n \geq 0$, $A^{n}(t)$ has a continuous linear extension form $\Phi_{q_{m}}$ into $\Phi_{q_{k}}$ (also denoted by $A^{n}(t)$) and the mapping $t \mapsto A^{n}(t)$ is $\mathcal{L}(\Phi_{q_{m}}, \Phi_{q_{k}})$-continuous.
\item For each $T>0$ there exists $k_{0} \geq 0$ and for $k \geq k_{0}$ there are constants $M_{k}=M_{k}(T) \geq 1$ and $\sigma_{k}= \sigma_{k}(T)$ such that for every $n \geq 0$,
\begin{equation}\label{eqDefiStableGenerators}
q_{k} \left(S^{n}_{t_{1}}(s_{1})S^{n}_{t_{2}}(s_{2}) \cdots S^{n}_{t_{m}}(s_{m}) \phi \right) \leq M_{k} \exp \left(  \sigma_{k} \sum_{j=1}^{m} s_{j} \right) q_{k}(\phi), \quad \forall \phi \in \Phi, \, s_{j} \geq 0, 
\end{equation} 
whenever $0 \leq t_{1} \leq t_{2} \leq \cdots \leq t_{m} \leq T$, $m \geq 0$. 
\end{enumerate}

By Theorem 1.3 in \cite{KallianpurPerezAbreu:1988} for each $n \geq 0$ there exists a unique $(C_{0},1)$-backward evolution system $(U^{n}(s,t): 0 \leq s \leq t < \infty)$ on $\Phi$ which is generated by the family $(A^{n}(t):t \geq 0)$. Assume moreover that the following is satisfied:

\begin{enumerate} \setcounter{enumi}{2}
\item For each $T>0$ and $k>0$ there exists $m > k$ such that
$$\lim_{n \rightarrow \infty} \sup_{0 \leq t \leq T} \norm{A^{n}(t)-A^{0}(t)}_{\mathcal{L}(\Phi_{q_{m}}, \Phi_{q_{k}})}=0. $$
\item For each $T>0$ and $k>k_{0}$ (with $k_{0}$ corresponding to $T$ as in $(2)$) we have
$$ \sup_{0 \leq s \leq t \leq T} \norm{U^{n}(s,t)}_{\mathcal{L}(\Phi_{q_{k}}, \Phi_{q_{k}})}< \infty. $$
\end{enumerate}
Then, one can show by following similar arguments to those used in the proof of Lemma 2.2 in \cite{KallianpurPerezAbreu:1989}  that (1)-(4) above imply that given $T>0$ for some $k \in \N$ we have 
$$ \sup_{0 \leq s \leq t \leq T} q_{k}(U^{n}(s,t)\phi-U^{0}(s,t)\phi) \rightarrow 0, \quad \mbox{ as } n \rightarrow \infty, \mbox{ for all } \phi \in \Phi,  $$
and 
$$ \sup_{0 \leq s \leq t \leq T} q_{k}(A^{n}(s)U^{n}(s,t)\phi-A^{0}(s)U^{0}(s,t)\phi) \rightarrow 0, \quad \mbox{ as } n \rightarrow \infty, \mbox{ for all } \phi \in \Phi. $$
Thus assumption (2) in Theorem \ref{theoAlmosSureConvOrnsUhleProce} is satisfied since convergence in the norm $q_{k}$ is stronger than convergence in $\Phi$.  

It is worth to mention that the sufficient conditions introduced above for the case of a nuclear Fr\'{e}chet space $\Phi$ can be modified to hold under the assumption that $\Phi$ is a (general) nuclear space. See Theorem 6.6 in \cite{Wu:1994}. 
\end{remark}

\smallskip

\textbf{Acknowledgements} The author thanks The University of Costa Rica for providing financial support through the grant ``821-C2-132- Procesos cil\'{i}ndricos y ecuaciones diferenciales estoc\'{a}sticas''.

\end{document}